\documentclass[11pt]{amsart}
\usepackage{amsmath,amssymb,euscript,mathrsfs,pstricks}
\usepackage{hyperref}
\usepackage[all]{xy}
\usepackage{enumerate}

\pushQED{\qed}

\psset{unit=1pt}
\psset{arrowsize=4pt 1}
\psset{linewidth=.5pt}

\newcommand{\define}{\textbf}

\renewcommand{\setminus}{\smallsetminus}
\renewcommand{\phi}{\varphi}

\renewcommand{\tilde}{\widetilde}
\renewcommand{\hat}{\widehat}
\renewcommand{\bar}{\overline}
\renewcommand{\wedge}{\bigwedge}

\newcommand{\C}{\mathbb{C}}
\newcommand{\Q}{\mathbb{Q}}
\newcommand{\R}{\mathbb{R}}
\newcommand{\N}{\mathbb{N}}
\newcommand{\Z}{\mathbb{Z}}
\renewcommand{\P}{\mathbb{P}}
\newcommand{\A}{\mathbb{A}}

\DeclareMathOperator{\Sym}{Sym}
\DeclareMathOperator{\codim}{codim}

\DeclareMathOperator{\Hom}{Hom}
\DeclareMathOperator{\Spec}{Spec}

\DeclareMathOperator{\orb}{orb}

\DeclareMathOperator{\Ind}{Ind}

\DeclareMathOperator{\Int}{Int}
\DeclareMathOperator{\Proj}{Proj}

\DeclareMathOperator{\st}{st}

\DeclareMathOperator{\GL}{GL}
\DeclareMathOperator{\aff}{aff}

\DeclareMathOperator{\conv}{conv}
\DeclareMathOperator{\Hilb}{Hilb}
\DeclareMathOperator{\SL}{SL}

\newtheorem{theorem}{Theorem}[section]

\newtheorem{lemma}[theorem]{Lemma}

\newtheorem{proposition}[theorem]{Proposition}

\newtheorem{corollary}[theorem]{Corollary}

\theoremstyle{definition}
\newtheorem{definition}[theorem]{Definition}
\newtheorem{remark}[theorem]{Remark}
\newtheorem{example}[theorem]{Example}
\newtheorem{question}[theorem]{Question}
\newtheorem{conjecture}[theorem]{Conjecture}

\newcommand{\excise}[1]{}

\begin{document}

\title{New mirror pairs of Calabi-Yau orbifolds}

\author{Alan Stapledon}
\address{Department of Mathematics\\University of British Columbia\\ BC, Canada V6T 1Z2}
\email{astapldn@math.ubc.ca}

\date{}
\thanks{}

\begin{abstract}
We prove a representation-theoretic version of Borisov-Batyrev mirror symmetry, and use it to construct infinitely many new pairs of orbifolds with mirror Hodge diamonds, with respect to the usual Hodge structure on singular complex cohomology. We conjecture that the corresponding orbifold Hodge diamonds are also mirror. When $X$ is the Fermat quintic in $\P^4$, and $\tilde{X}^*$ is a $\Sym_5$-equivariant, toric resolution of its mirror $X^*$, we deduce that for any subgroup $\Gamma$ of the alternating group $A_5$, the $\Gamma$-Hilbert schemes 
 $\Gamma$-$\Hilb(X)$ and $\Gamma$-$\Hilb(\tilde{X}^*)$ are smooth Calabi-Yau threefolds with (explicitly computed) mirror Hodge diamonds. 
\end{abstract}

\maketitle
\tableofcontents

\section{Introduction}

Based originally on the computations of physicists in \cite{CLSCalabi}, mirror symmetry predicts that $n$-dimensional, complex, Calabi-Yau manifolds occur in pairs $(V, W)$ with mirror Hodge diamonds. That is,
\[
h^{p,q}(V) = h^{n - p, q} (W) \textrm{  for } 0 \le p,q \le n. 
\]
The work of Batyrev and Dais \cite{BatStringy, BatNon, BDStrong} led to the more general conjecture that $n$-dimensional, possibly singular, complex, Calabi-Yau varieties occur in pairs $(V, W)$ with \emph{stringy invariants} satisfying  the relation
\[
E_{\st}(V;u,v) = (-u)^n E_{\st}(W;u^{-1},v).
\]
If $\tilde{V} \rightarrow V$ and $\tilde{W} \rightarrow W$ are crepant resolutions of $V$ and $W$ respectively, then this  says that 
\[
h^{p,q}(\tilde{V}) = h^{n - p, q} (\tilde{W}) \textrm{  for } 0 \le p,q \le n. 
\]
Mirror pairs of Calabi-Yau manifolds are constructed in the work of Borcea \cite{BorK3}, Voisin \cite{VoiMiroirs} and Batyrev, Ciocan-Fontanine, Kim and van Straten \cite{BCKvSConifold, BCKvSMirror}. Conjectural mirror pairs also appear in the work  of R{\o}dland\cite{RodPfaffian}, B\"ohm \cite{BoeMirror}  and Kanazawa \cite{KanPfaffian}. Apart from these examples, all known mirror pairs of Calabi-Yau varieties appear as a result of a general construction of Batyrev and Borisov \cite{BBMirror} of mirror pairs of complete intersections in Fano toric varieties. In the hypersurface case, their construction is as follows:  they observe that $d$-dimensional \emph{reflexive} lattice polytopes $P$ and $P^*$ naturally appear in pairs associated to dual lattices $M$ and $N$ respectively, and let $X$ and $X^*$ be hypersurfaces in the associated complex toric varieties that are non-degenerate with respect to $P$ and $P^*$ respectively in the sense of Khovanski{\u\i} \cite{KhoNewton}. 

Let $\Gamma$ be a finite group that acts linearly on $M$, and leaves $P$ invariant. Then one can show that this induces an action of $\Gamma$ on the toric varieties associated to $P$ and $P^*$ respectively. 
Assume that the hypersurfaces $X$ and $X^*$ are $\Gamma$-invariant. In Definition~\ref{d:esi}, we define the \emph{equivariant stringy invariant} $E_{\st, \Gamma}(Z;u,v)$ of  a complex, Gorenstein variety $Z$ with an action of $\Gamma$. The invariants $E_{\st, \Gamma}(X;u,v)$ and $E_{\st, \Gamma}(X^*;u,v)$ are polynomials in $u$ and $v$ with coefficients in the complex representation ring $R(\Gamma)$ of $\Gamma$ (see Corollary~\ref{c:reflexive}). In Theorem~\ref{t:mirror}, we prove the following representation-theoretic version of Batyrev-Borisov mirror symmetry that was conjectured  in \cite[Conjecture~9.1]{YoRepresentations}, 
\begin{equation}\label{e:mirroreq}
E_{\st, \Gamma}(X; u,v) = (-u)^{d - 1} \det(\rho) \cdot  E_{\st, \Gamma}(X^*; u^{-1},v),
\end{equation}
where $\det(\rho)$ denotes the determinant representation associated with the action of $\Gamma$ on $M$. 
In particular,  it follows that if there exist $\Gamma$-equivariant, crepant, toric resolutions 
$\tilde{X} \rightarrow X$ and $\tilde{X}^* \rightarrow X^*$, then we have an equality of representations
\[
H^{p,q}(\tilde{X}) = \det(\rho) \cdot H^{d - 1 - p,q}(\tilde{X}^*)  \in R(\Gamma) \: \textrm{  for  } \: 0 \le p,q \le d - 1. 
\]
The result above is deduced as a consequence of a general formula for the equivariant Hodge-Deligne polynomial of a non-degenerate hypersurface in a torus (Theorem~\ref{t:formula}), which may be considered the main result of the paper, together with a formula for the equivariant stringy invariant (Proposition~\ref{p:stringy}), and a simplification for 
 hypersurfaces that are non-degenerate with respect to a reflexive polytope (Corollary~\ref{c:reflexive}). 

We remark that in the simplest case, when $\Gamma$ is trivial, Batyrev and Borisov's proof of \eqref{e:mirroreq} relies on some deep results on intersection cohomology. As remarked by Borisov in \cite[Section~5]{BMString}:

\emph{`However, it was very difficult to compute the Hodge-Deligne numbers of an arbitrary non-degenerate affine hypersurface. This was a major technical problem in the proof of mirror symmetry of the stringy Hodge numbers for Calabi-Yau complete intersections in \cite{BBMirror}.'}

On the other hand, after developing some  combinatorial machinery in Section~\ref{s:cones} and Section~\ref{s:tilde}, and using the results of \cite{YoEquivariant} and \cite{YoRepresentations}, we are able to provide a purely combinatorial proof of our result. In the case when $\Gamma$ is trivial, the author has been informed that unpublished combinatorial proofs were independently given and subsequently lost by Borisov and Khovanski{\u\i}. 

As an immediate corollary, we have the following surprising result.  If $\Gamma \subseteq \SL(M)$, then the orbifolds $\tilde{X}/\Gamma$ and $\tilde{X}^*/\Gamma$, which are possibly singular but whose cohomology admits a pure Hodge structure, have mirror Hodge diamonds. That is, 
\[
h^{p,q}(\tilde{X}/\Gamma) = h^{d - 1 - p,q}(\tilde{X}^*/\Gamma)  \: \textrm{  for  } \: 0 \le p,q \le d - 1. 
\]
When $\Gamma$ acts freely on $\tilde{X}$ and $\tilde{X}^*$, this produces new mirror pairs of Calabi-Yau manifolds (see Section~\ref{s:central} for explicit examples).  On the other hand, when the orbifolds have singularities, 
observe that this is a strictly different statement to the usual mirror symmetry test. Also, observe that these orbifolds do not appear in Batyrev and Borisov's construction; the singularities of these orbifolds can be non-abelian, whereas all varieties in Batyrev and Borisov's construction have toroidal, and hence abelian, singularities.

Results of Yasuda in \cite{YasMotivic} imply that the coefficients of $E_{\st}(V;u,v)$ for an orbifold $V$ are  equal to the \emph{orbifold Hodge numbers} $h^{p,q}_{\orb}(V)$ of $V$, as introduced by Chen and Ruan in \cite{CRNew}. The above results suggest the following McKay-type correspondence; if $\Gamma \subseteq \SL(M)$, then the orbifolds  $\tilde{X}/\Gamma$ and $\tilde{X}^*/\Gamma$ should have mirror orbifold Hodge diamonds. That is,
\[
h^{p,q}_{\orb}(\tilde{X}/\Gamma) = h^{d - 1 - p,q}_{\orb}(\tilde{X}^*/\Gamma)  \: \textrm{  for  } \: 0 \le p,q \le d - 1. 
\]
In particular, if there exist crepant resolutions $Z \rightarrow \tilde{X}/\Gamma$ and $Z^* \rightarrow \tilde{X}^*/\Gamma$, then this would imply that $Z$ and $Z^*$ are $(d - 1)$-dimensional Calabi-Yau manifolds with mirror Hodge diamonds. We refer the reader to Conjecture~\ref{c:Fantechi} for a precise, and more general, statement. 

The quintic threefold $X = \{ x_0^5 + x_1^5 + x_2^5 + x_3^5  + x_4^5 = 0 \} \subseteq \P^4$
admits an action of the symmetry group $\Sym_5$ by permuting co-ordinates. 
 A $\Sym_5$-invariant mirror hypersurface $X^*$ in Batyrev and Borisov's construction is singular and admits a $\Sym_5$-equivariant, toric, crepant resolution $\tilde{X}^* \rightarrow X^*$.
 In fact, this was one of the first and most significant examples in mirror symmetry. That is,
 a remarkable, early application of mirror symmetry was a prediction by physicists in \cite{COGPPair} of the number of rational curves of a given degree on $X$. A mathematical explanation and interpretation was later given by Morrison in \cite{MorMirror}. 
 
 Our results state that for any subgroup $\Gamma$ of the group $A_5$ of even permutations, $X/\Gamma$ and 
 $\tilde{X}^*/\Gamma$ are orbifolds with mirror Hodge diamonds. In Section~\ref{s:quintic}, we explicitly verify Conjecture~\ref{c:Fantechi}, and deduce that $X/\Gamma$ and 
 $\tilde{X}^*/\Gamma$ have mirror orbifold  Hodge diamonds. Moreover, a theorem of Bridgeland, King and Reid \cite[Theorem~1.2]{BKRMcKay} implies that  crepant resolutions of $X/A_5$ and $\tilde{X}^*/A_5$ are given by the $\Gamma$-Hilbert schemes 
 $\Gamma$-$\Hilb(X)$ and $\Gamma$-$\Hilb(\tilde{X}^*)$ respectively, which parametrize $0$-dimensional subschemes $Z$ such that the induced representation of $\Gamma$ on $H^0(Z, \mathcal{O}_Z)$ is isomorphic to the regular representation $\C[\Gamma]$ of $\Gamma$.
We deduce that $\Gamma$-$\Hilb(X)$ and $\Gamma$-$\Hilb(\tilde{X}^*)$ are Calabi-Yau manifolds with mirror Hodge diamonds. For example, when $\Gamma = A_5$, the respective Hodge diamonds are 
\begin{center}
\begin{tabular}{ c    c     c   c  c  c  c                 c c  c                               c    c     c   c  c  c  c         }
 &  &  & $1$ & &                                             & &&      &        &  &  & $1$ & &   \\
  &  & $0$ & & $0$ &  &                                & &&           &  & $0$ & & $0$ &  &\\
   &  $0$  &  & $5$ & & $0$ &                    &  &&         &  $0$  &  & $15$ & & $0$ &\\
     $1$  &   & $15$  &  & $15$ &  & $1$      & &&        $1$  &   & $5$  &  & $5$ &  & $1$. \\
      &  $0$  &  & $5$ & & $0$ &              &&   &       &  $0$  &  & $15$ & & $0$ &  \\
        &  & $0$ & & $0$ &  &                         &&   &          &  & $0$ & & $0$ &  &\\
        &  &  & $1$ & &                               &          &   &&            &  &  & $1$ & & \\
\end{tabular}
 \end{center}

Finally, we expect the results of Sections \ref{s:cones}, \ref{s:tilde}, \ref{s:formula} and 
\ref{s:stringy} to have independent interest outside of applications to mirror symmetry (see, for example, Remark~\ref{r:geometric}, Example~\ref{e:cc}, Remark~\ref{r:toric} and Remark~\ref{r:motivic}). We also expect that similar results hold for complete intersections, rather than hypersurfaces, although we do not pursue this here. 

\medskip
\noindent
{\it Notation and conventions.}  
All varieties are over the complex numbers, and all cohomology will be taken 
 with complex coefficients, with respect to the usual (complex) topology. All group actions will be left group actions. If $\Gamma$ is a finite group, then $R(\Gamma)$ denotes the complex representation ring of $\Gamma$. We will often identify a virtual representation $\chi$ in $R(\Gamma)$ with its associated virtual character, and write $\chi(\gamma)$ for its evaluation at $\gamma$ in $\Gamma$.  If $M$ is a lattice, then we write $M_\R  := M \otimes_\Z \R$.

\medskip
\noindent
{\it Acknowledgements.}  
In the case when $\Gamma$ is trivial, a significant amount of the work in Section~\ref{s:formula} was completed while the author was visiting the IAS in 2006, where he benefited enormously  from the help and encouragement of Mircea Musta\c t\v a and Lev Borisov. The author would also like to thank Igor Dolgachev for pointing him to known references of mirror Calabi-Yau's, and would like to thank 
Askold Khovanski{\u\i} for useful conversations.


\section{Representations and cones}\label{s:cones}

The goal of this section is to introduce and study representation-theoretic analogues of the $h$-polynomial and $g$-polynomial of a polyhedral cone. 

We will use the following setup throughout this section. Let $\Gamma$ be a finite group, and let 
$\rho: \Gamma \rightarrow \GL_{d + 1}(\R)$ be a real representation. Let $C \subseteq \R^{d + 1}$ be a $(d + 1)$-dimensional, pointed, polyhedral, $\Gamma$-invariant cone. For each face $F$ of $C$, let $\Gamma_F$ denote the stabilizer of $F$ with complex representation ring $R(\Gamma_F)$, and let $\rho_F: \Gamma_F \rightarrow \GL_{\dim F}(\R)$ denote the representation of $\Gamma_F$ on the linear span of $F$.  Let $\det (\rho_F) : \Gamma_F \rightarrow \{ \pm 1 \} \subseteq \R$ be the corresponding determinant representation. 
When $F = \{ 0 \}$, we set $\rho_F$ to be the trivial representation of $\Gamma$. Fix $\gamma \in \Gamma$, and let $B_\gamma$ denote the poset of $\gamma$-invariant faces of $C$.

If $B$ is a finite poset then
the M\"obius function $\mu_{B}: B \times B \rightarrow \mathbb{Z}$ is defined recursively as follows (see, for example, \cite[Section~3.7]{StaEnumerative}),
\[
\mu_{B}(x,y) =  \left\{ \begin{array}{ll}
1 & \textrm{  if  } x = y \\
0 & \textrm{  if  } x > y \\
- \sum_{x < z \leq y} \mu_{B}(z,y) = - \sum_{x \leq z < y} \mu_{B}(x,z) & \textrm{  if  } x < y 
\end{array} \right.
,
\]
and satisfies the property that for any function $h: B \rightarrow A$ to an abelian group $A$, 
\begin{equation}\label{inversion}
h(x) = \sum_{x \leq y} \mu_{B}(x,y) g(y), \textrm{ where } g(y) = \sum_{y \leq z} h(z).
\end{equation}
For any pair $z \leq x$ in $B$, we can consider the interval $[z,x] = \{ y \in B \mid z \leq y \leq x \}$. 
Suppose that $B$ has a minimal element $0$ and a maximal element $1$, and 
that every maximal chain in $B$ has the same length. The \emph{rank} $r(x)$ of an element $x$ in $B$ is equal to the length of a maximal chain in $[0,x]$, and the rank of $B$ is $r(1)$. In this case, we say that $B$ is \emph{Eulerian} if $\mu_{B}(x,y) = (-1)^{r(x) - r(y)}$ for $x \leq y$. 


\begin{example}\cite{ZieLectures}\label{polytope}
The poset of faces of a pointed, polyhedral cone $F$ is an Eulerian poset under inclusion with rank function 
$r(F) = \dim F$.
\end{example}

The following lemma will be key in proving the results of this section. 

\begin{lemma}\label{l:mobius}
For fixed $\gamma \in \Gamma$, the poset $B_\gamma$ of $\gamma$-invariant faces of $C$ is an Eulerian poset, with M\"obius function $\mu_\gamma$ given by
\[
\mu_{\gamma}(F, F') = (-1)^{\dim F' - \dim F} \det \rho_F (\gamma) \det \rho_{F'} (\gamma),
\]
where $F \subseteq F'$ are $\gamma$-invariant faces of $C$.  
\end{lemma}
\begin{proof}
Consider the linear subspace $L_\gamma = \{ x \in \R^{d + 1} \mid \gamma \cdot x = x \}$ and the cone $C_\gamma  = C \cap L_\gamma$.   Observe that every $\gamma$-invariant face of $C$ contains a $\gamma$-fixed point in its relative interior. Indeed, we can construct such a point by 
summing the vectors in the $\gamma$-orbit of any fixed interior point.
It follows that  the elements 
  of $B_\gamma$ are in inclusion-preserving bijection with the faces of $C_\gamma$, where a $\gamma$-invariant face $F$ of $C$ corresponds to the face $F_\gamma = F \cap L_\gamma$ of $C_\gamma$. 
  Moreover, Lemma~5.5 and Remark~5.6 in \cite{YoEquivariant} imply that
  \[
  \det \rho_F(\gamma) = (-1)^{\dim F - \dim F_\gamma}. 
  \]
  Hence Example~\ref{polytope} implies that $B_{\gamma}$ is an Eulerian poset with M\"obius function given by
  \[
  \mu_{\gamma}(F, F') = (-1)^{\dim F_\gamma' - \dim {F}_\gamma} = (-1)^{\dim F' - \dim F} \det \rho_F (\gamma) \det \rho_{F'} (\gamma).
  \]
\end{proof}

We will also need the following lemma.

\begin{lemma}\label{l:once}
Fix $\gamma \in \Gamma$, and let $F$ be a non-zero $\gamma$-invariant face of $C$. Then
\[
(-t)^{\dim F} \det(t^{-1}I - \rho_F(\gamma)) = \det \rho_F(\gamma) \det(tI - \rho_F(\gamma)). 
\]
\end{lemma}
\begin{proof}
Since $\gamma$ has finite order, we may assume that $\gamma$ acts on the linear span of $F$ via a diagonal matrix
$(\lambda_1, \ldots, \lambda_{\dim F})$ whose entries are roots of unity. Using the fact that both sides of the equation above are real-valued polynomials, the left hand side equals
\begin{align*}
(-t)^{\dim F}(t^{-1} - \lambda_1)\cdots (t^{-1} - \lambda_{\dim F}) &= 
\lambda_1 \cdots \lambda_{\dim F} (t - \bar{\lambda}_1)\cdots (t - \bar{\lambda}_{\dim F})  \\
&= \det \rho_F(\gamma)  (t - \lambda_1)\cdots (t - \lambda_{\dim F}).
\end{align*}
\end{proof}

Stanley introduced the $h$-polynomial and $g$-polynomial of an Eulerian poset in \cite{StaGeneralized}. Our next goal is to recursively define two polynomials of virtual representations associated to the action of $\Gamma$ on $C$, which may be viewed as representation-theoretic analogues of the $h$-polynomial and $g$-polynomial of the poset of faces of $C$. 

More specifically, if 
$F$ is a non-zero face of $C$,
consider the polynomial of virtual representations
\begin{equation}\label{e:determinant}
\det[tI - \rho_F] := \sum_{i = 0}^{\dim F} (-1)^{\dim F - i}(\wedge^{\dim F - i}\rho_F) t^{i}  \in R(\Gamma_F)[t],
\end{equation}
where $\wedge^{j}\rho_F$ denotes the $j^{\textrm{th}}$ exterior product of the representation $\rho_F$. Observe that the evaluation of the associated character  $\det[tI - \rho_F](\gamma)$ at $\gamma \in \Gamma$ is equal to  $\det(tI - \rho_F(\gamma))$.  If $F = \{ 0 \}$, then we set $\det[tI - \rho_F] = 1 \in R(\Gamma)$. Induction of  representations from $\Gamma_F$ to $\Gamma$ gives rise to an additive homomorphism
$\Ind_{\Gamma_F}^{\Gamma}: R(\Gamma_F)[t] \rightarrow R(\Gamma)[t]$.

\begin{definition}\label{d:hg}
 If $C = \{0 \}$, then define $H(C, t) = G(C, t) \in R(\Gamma)[t]$ to be the trivial representation. If
 $C \ne \{0 \}$, then define elements $H(C, t)$ and $G(C,t)$ of $R(\Gamma)[t]$ recursively as follows:
\[
H(C,t) = \sum_{\substack{[F] \in C/\Gamma \\ F \ne C }} \det[tI - \rho] \Ind_{\Gamma_F}^{\Gamma}
\frac{G(F,t)}{(t - 1)  \det[tI - \rho_F] },
\]
where $C/\Gamma$ denotes the set of $\Gamma$-orbits of faces of $C$,  and
\[
G(C, t) = \tau_{\le \frac{\dim C - 1}{2}} (1 - t) H(C, t), 
\]
where $\tau_{\le i}$ denotes truncation of all terms of degree at most $i$. 
\end{definition}

\begin{remark}\label{r:character}
Assuming $C \ne \{ 0 \}$, if one evaluates the virtual characters of the virtual representations above at $\gamma \in \Gamma$, one obtains 
\[
H(C,t)(\gamma) = \frac{\det(tI - \rho(\gamma))}{t - 1} \sum_{\substack{\gamma \cdot F = F \\ F \ne C}} 
\frac{G(F,t)(\gamma)}{\det(tI - \rho_F(\gamma)) },
\]
where $\det(tI - \rho_F(\gamma)) = 1$ when $F = \{ 0 \}$, and 
\[
G(C, t)(\gamma) = \tau_{\le \frac{\dim C - 1}{2}} (1 - t) H(C, t)(\gamma). 
\]
Since $\gamma$ fixes a non-zero vector in the interior of $C$, it 
 follows that $\frac{\det(tI - \rho(\gamma))}{(t - 1)\det(tI - \rho_F(\gamma))}$ is a polynomial in $t$ of degree $d  - \dim F$. We deduce that $H(C,t)$ is a polynomial of degree $d$ with leading coefficient equal to $1 \in R(\Gamma)$. 
 Observe that
 evaluation of characters at $1 \in \Gamma$ yields the usual $h$-polynomial and $g$-polynomial of the poset of faces of $C$ \cite{StaGeneralized}. 
\end{remark}

\begin{remark}\label{r:geometric}
The equations in the definition above  were motivated by
the following geometric interpretation of $H(C,t)$ and $G(C,t)$, which extends known interpretations of the $h$-polynomial and $g$-polynomial when $\Gamma = \{1 \}$. In the latter case, we refer the reader to 
\cite{BraRemarks} for an excellent survey paper on the combinatorial intersection cohomology of fans.

Firstly, observe that $C$ may be viewed as the cone over a $\Gamma$-invariant polytope $P$. 
Indeed, $\Gamma$ acts on the dual cone $\check{C}$ of $C$, and one may construct a $\Gamma$-invariant point $u$ in the interior of $\check{C}$ by summing the elements of a $\Gamma$-orbit of any interior point. The intersection of $C$ with an appropriate affine translate $H$ of the $\Gamma$-invariant hyperplane 
determined by $u$ is a   $\Gamma$-invariant polytope $P$. Moreover, one may fix a $\Gamma$-invariant point $v$ in the relative interior of $P$.
 
Let $\Sigma$ denote the $\Gamma$-invariant fan over the faces of $P$ in $H - v$. 
Then $\Gamma$ acts on the global sections $IH^{2i}(\Sigma)$ of the intersection cohomology sheaf of $\Sigma$, and $H(C,t) = \sum_{i = 0}^d IH^{2i}(\Sigma) t^i \in R(G)$. In particular, $H(C,t)$ is a polynomial of representations (rather than virtual representations) of $\Gamma$, and Poincar\'e duality for intersection cohomology sheaves on fans implies that $H(C,t) = t^d H(C,t^{-1})$. The Hard Lefschetz theorem then implies that 
$G(C,t)$ is a polynomial of representations. 

We will not need this remark in what follows, so we leave the proof open. In Proposition~\ref{p:symmetry}, we give a combinatorial proof that $H(C,t) = t^d H(C,t^{-1})$. 
\end{remark}

\begin{example}
Suppose that $C$ is a simplicial cone i.e. $C$ has precisely $d + 1$ rays. 
Then it follows from Remark~\ref{r:geometric} that the coefficients of $H(C,t)$ are the representations of $\Gamma$ on the cohomology of $\P^d$, and hence 
$H(C,t) = 1 + t + \cdots + t^d$ and $G(C,t) = 1$.  
One may also deduce this from the definition. 
Indeed, in this case $\rho$ is  the permutation representation of $\Gamma$ acting on the rays of $C$. For any $\gamma \in \Gamma$, let $I_1, \ldots, I_s$ denote the  $\gamma$-orbits of rays of $C$.  
Then the $\gamma$-invariant faces $F_J$ of $C$ are precisely the faces spanned by the rays in a subset
$J \subseteq \{ 1, \ldots , s \}$
 of the $\gamma$-orbits of rays of $C$. Using induction, we compute
 \begin{align*}
 H(C,t)(\gamma) = \sum_{J \subsetneq \{ 1, \ldots , s \}} 
\frac{ \prod_{j \notin J} (t^{|I_j|} - 1)}{t - 1 } &=  
\frac{ \prod_{J \subseteq \{ 1, \ldots , s \}}  ((t^{|I_j|} - 1) + 1) - 1}{t - 1 }  \\
&= 1 + t + \cdots + t^d.
 \end{align*}
\end{example}

\begin{example}\label{e:cc}
Let $\Gamma = \Z_2 = \{ 1, \epsilon \}$ act on $\R^d$ sending $v$ to $-v$, and let $P$ be a $d$-dimensional 
$\Z_2$-invariant polytope. In this case, $P$ is called \emph{centrally symmetric}. 
If $C$ denotes the cone over $P \times 1$ in $\R^{d + 1}$, then $C$ has no proper, non-zero $\epsilon$-invariant faces. By Remark~\ref{r:character}, $H(C,t)(1) = h(C,t) = \sum_{i = 0}^d h_i t^i$ is the usual $h$-vector of $C$, and $H(C,t)(\epsilon) = (1 + t)^d$. If $\zeta$ denotes the non-trivial character of $\Z_2$, we deduce that 
\[
H(C,t) = \frac{h(C,t) + (1 + t)^d}{2} +   \frac{h(C,t) - (1 + t)^d}{2}\zeta.  
\] 
In particular, the Hard Lefschetz theorem (see Remark~\ref{r:geometric}) implies that $h_i - h_{i - 1} \ge \binom{d}{i} - \binom{d}{i - 1}$ for $1 \le i \le \lfloor \frac{d}{2} \rfloor$, a result due to A'Campo-Neuen
 \cite{ACaToric}.
\end{example}

\begin{example}
Let $\Gamma = \Sym_3$ act on $V = \R^3$ by the standard representation, and let $P = [-1,1]^3$ be a $\Gamma$-invariant polytope. 
If $C$ denotes the cone over $P \times 1$ in $\R^4$, then  one 
computes that $H(C,t) = 1 + (2 + V)t + (2 + V)t^2 + t^3$ and $G(C,t) = 1 + (1 + V)t$. 
\end{example}

When $\Gamma$ is trivial, the result below holds, more generally,  for $h$-polynomials and $g$-polynomials of Eulerian posets \cite[Theorem~2.4]{StaGeneralized}.

\begin{proposition}\label{p:symmetry}
With the notation above, $H(C,t) = t^d H(C, t^{-1})$. Equivalently,
\[
t^{d + 1} G(C,t^{-1}) =  \sum_{ [F] \in C/\Gamma}
\det[tI - \rho]  \Ind^{\Gamma}_{\Gamma_F} \frac{G(F,t)}{\det[tI - \rho_F]},
\]
where $C/\Gamma$ denotes the set of $\Gamma$-orbits of faces of $C$.
\end{proposition}
\begin{proof}
By definition, the second statement is equivalent to 
\[
t^{d + 1} G(C,t^{-1}) - G(C,t) = (t - 1)H(C,t),
\]
which is equivalent to $H(C,t) = t^d H(C, t^{-1})$. 
For any $\gamma \in \Gamma$,
\begin{equation*}
t^d H(C, t^{-1})(\gamma) = \frac{t^d \det(t^{-1}I - \rho(\gamma))}{t^{-1} - 1} \sum_{\substack{\gamma \cdot F = F \\ F \ne C}} 
\frac{G(F,t^{-1})(\gamma)}{\det(t^{-1}I - \rho_F(\gamma)) },  
\end{equation*}
where $\det(t^{-1}I - \rho_F(\gamma)) = 1$ if $F = \{ 0 \}$. By Lemma~\ref{l:once}, the latter sum equals
\[
 \frac{\det(tI - \rho(\gamma))}{t - 1} \sum_{\substack{\gamma \cdot F = F \\ F \ne C}} 
\frac{(-1)^{d - \dim F} \det \rho(\gamma) \det \rho_F(\gamma)  t^{\dim F}G(F,t^{-1})(\gamma)}{\det(tI - \rho_F(\gamma)) },
\]
where $\det(tI - \rho_F(\gamma)) = 1$ if $F = \{ 0 \}$. 
By Lemma~\ref{l:mobius}, we conclude that 
\[
 t^d H(C, t^{-1})(\gamma)  = \frac{-\det(tI - \rho(\gamma))}{t - 1} \sum_{\substack{\gamma \cdot F = F \\ F \ne C}} 
\frac{\mu_\gamma(F,C)  t^{\dim F}G(F,t^{-1})(\gamma)}{\det(tI - \rho_F(\gamma)) }.
\]
By induction on dimension, the latter sum equals
 \begin{align*}
 &\frac{-\det(tI - \rho(\gamma))}{t - 1} \sum_{\substack{\gamma \cdot F = F \\ F \ne C}} 
\frac{\mu_\gamma(F,C)}{\det(tI - \rho_F(\gamma)) }\sum_{\substack{\gamma \cdot F' = F' \\ F' \subseteq F}} \frac{ \det(tI - \rho_F(\gamma))G(F',t)(\gamma) }{ \det(tI - \rho_{F'}(\gamma)) } \\
&= \frac{-\det(tI - \rho(\gamma))}{t - 1} \sum_{\substack{\gamma \cdot F' = F' \\ F' \ne C}} 
 \frac{ G(F',t)(\gamma) }{ \det(tI - \rho_{F'}(\gamma)) }
  \sum_{       \substack{\gamma \cdot F = F \\ F' \subseteq F \ne C}  } \mu_\gamma(F,C) \\
  &=  \frac{\det(tI - \rho(\gamma))}{t - 1} \sum_{\substack{\gamma \cdot F' = F' \\ F' \ne C}} 
 \frac{ G(F',t)(\gamma) }{ \det(tI - \rho_{F'}(\gamma)) }  = H(C,t)(\gamma). 
\end{align*}
\end{proof}

We will need the following two lemmas. Recall that $\Gamma$ acts on the dual cone $\check{C}$ of $C$, and there is an inclusion reversing bijection between the faces $F$ of $C$ and the faces $F^*$ of 
 $\check{C}$ such that $\dim F + \dim F^* = d + 1$.

\begin{lemma}\label{l:ratio}
If $F \subseteq F'$ are $\gamma$-invariant faces, then
\[
\frac{ \det(tI - \rho_{F'}(\gamma)) }{ \det(tI - \rho_{F}(\gamma))  } = \frac{ \det(tI - \rho_{F^*}(\gamma)) }{ \det(tI - \rho_{(F')^*}(\gamma))  } 
\] 
\end{lemma}
\begin{proof}
Let $L_F$ and $L_{F'}$ denote the linear spans of $F$ and $F'$ respectively. If $\gamma$ acts on 
$L_{F'}/L_F$ with eigenvalues $\{ \lambda_1, \ldots, \lambda_r \}$, then  $\gamma$ acts on 
$L_{F^*}/L_{(F')^*}$ with eigenvalues $\{ \lambda_1^{-1}, \ldots, \lambda_r^{-1} \}$. Since $\rho$ is a real representation, 
\[
\{ \lambda_1^{-1}, \ldots, \lambda_r^{-1} \} = \{ \bar{\lambda}_1, \ldots, \bar{\lambda}_r \} = 
\{ \lambda_1, \ldots, \lambda_r \}.
\]
\end{proof}

In the case when $\Gamma$ is trivial, a version of the lemma below for Eulerian posets is proved by Stanley in \cite[Corollary 8.3]{StaSubdivisions}. 
If $F$ is a face of $C$ with linear span $L_F$, then let $C/F$ denote the projection of $C$ to $\R^{d + 1}/L_F$.

\begin{lemma}\label{l:convolution}
With the notation above, if $C$ is a non-zero cone, then 
\[
\sum_{[F] \in C/\Gamma} (-1)^{\dim F} \Ind_{\Gamma_F}^{\Gamma} \det(\rho_F) G(C/F,t)G(\check{C}/F^*,t) = 0,
\]
where $C/\Gamma$ denotes the set of $\Gamma$-orbits of faces of $C$.
\end{lemma}
\begin{proof}
Fix $\gamma \in \Gamma$. For any $\gamma$-invariant faces 
$F' \subseteq F''$ of $C$,  let $\Phi_{F',F''}(t)$ be
\begin{equation*}
\sum_{  \substack{ \gamma \cdot F = F  \\ F' \subseteq F \subseteq F''         }}     (-1)^{\dim F} \det \rho_F(\gamma) G(F''/F,t)(\gamma)G((F')^*/F^*,t)(\gamma). 
\end{equation*}
We will prove the stronger claim that 
\[
\Phi_{F',F''}(t) = \left\{\begin{array}{cl}
 (-1)^{\dim F'} \det \rho_{F'}(\gamma) & \text{if } F' = F''  \\ 0 & \text{otherwise}. \end{array}\right. 
\]
The lemma then follows by setting $F' = \{ 0 \}$ and $F'' = C$. We proceed by induction on $\dim F'' - \dim F'$. Observe that the claim follows from the definitions when $F' = F''$, and hence we may assume that 
$F' \ne F''$. It follows from the definition of the $G$-polynomial that the degree of $\Phi_{F',F''}(t)$ is bounded by $\frac{ \dim F'' - \dim F' - 1 }{2}$. Hence it will be enough to show that $\Phi_{F',F''}(t) = 
t^{\dim F'' - \dim F' }\Phi_{F',F''}(t^{-1})$. We write $t^{\dim F'' - \dim F' }\Phi_{F',F''}(t^{-1})$ as
\[
\sum_{  \substack{ \gamma \cdot F = F  \\ F' \subseteq F \subseteq F''         }}     (-1)^{\dim F} \det \rho_F(\gamma) t^{\dim F''/F}G(F''/F,t^{-1})(\gamma)t^{\dim (F')^*/F^*}G((F')^*/F^*,t^{-1})(\gamma). 
\]
By Proposition~\ref{p:symmetry}, the latter expression is equal to 
\begin{align*}
\sum_{  \substack{ \gamma \cdot F = F  \\ F' \subseteq F \subseteq F''         }}     (-1)^{\dim F} \det \rho_F(\gamma) 
&\sum_{  \substack{ \gamma \cdot H'' = H'' \\ F \subseteq H'' \subseteq F''         }}  \det(tI - \rho_{F''/H''}(\gamma)) G( H''/F,t)(\gamma) \times \\
&\sum_{  \substack{ \gamma \cdot H' = H' \\ F' \subseteq H' \subseteq F        }}  \det(tI - \rho_{ (F')^*/(H')^*}(\gamma)) G( (H')^*/ F^*,t)(\gamma). 
\end{align*}
Rearranging gives
\[
\sum_{ F' \subseteq H' \subseteq H'' \subseteq F''   } \det(tI - \rho_{F''/H''}(\gamma))\det(tI - \rho_{ (F')^*/(H')^*}(\gamma)) \Phi_{H',H''}(t).
\]
By induction and Lemma~\ref{l:ratio}, the latter sum equals 
\[
 \Phi_{F',F''}(t) + \det(tI - \rho_{F''/F'}(\gamma)) \sum_{  \substack{ \gamma \cdot H = H \\ F' \subseteq H \subseteq F''        }  } (-1)^{\dim H} \det \rho_{H}(\gamma).
\]
By Lemma~\ref{l:mobius}, this simplifies to  $\Phi_{F',F''}(t)$, as desired.

\excise{
If $\Phi(t)$ denotes the left hand side of the equation above, then, by definition, the degree of $\Phi(t)$ is bounded by $d/2$.  Hence it will be enough to show that $\Phi(t) = t^{d  + 1}\Phi(t^{-1})$. For any $\gamma \in \Gamma$, we compute that $t^{d  + 1}\Phi(t^{-1})(\gamma)$ equals 
\[
 \sum_{\gamma \cdot F = F} (-1)^{\dim F} \det \rho_F(\gamma) t^{\dim F}G(F,t^{-1})(\gamma)t^{\dim F^*}G(F^*,t^{-1})(\gamma) 
 \]
 Using Proposition~\ref{p:symmetry} and Lemma~\ref{l:ratio}, the latter sum equals 
\begin{align*}
\sum_{\gamma \cdot F = F} (-1)^{\dim F} \det \rho_F(\gamma) 
  &\sum_{       \substack{\gamma \cdot F' = F' \\ F' \subseteq F}  } \frac{ \det(tI - \rho_{F}(\gamma))G(F',t) }{  
  \det(tI - \rho_{F'}(\gamma)) } \times \\
  &  \sum_{       \substack{\gamma \cdot F'' = F'' \\ F \subseteq F''}  } \frac{ \det(tI - \rho_{F''}(\gamma))G((F'')^*,t) }{  
  \det(tI - \rho_{F}(\gamma)) }.
\end{align*}
By Lemma~\ref{l:mobius}, we get
\[
\sum_{  \substack{     \gamma \cdot F' = F'       \\     \gamma \cdot F'' = F''       }   }
\frac{             \det(tI - \rho_{F''}(\gamma))  (-1)^{\dim F''} \det \rho_{F''}(\gamma) G(F',t) G((F'')^*,t)               }
{            \det(tI - \rho_{F'}(\gamma))                      } \alpha(F', F''), 
\]
where 
\[
\alpha(F', F'') = \sum_{  \substack{ \gamma \cdot F = F  \\ F' \subseteq F \subseteq F'' }}  \mu_\gamma(F, F'') =   \left\{\begin{array}{cl}
1 & \text{if } F' = F''  \\ 0 & \text{otherwise}. \end{array}\right. .
\]
We conclude that $  t^{d  + 1}\Phi(t^{-1})(\gamma) = \Phi(t)(\gamma)$, as desired. 
}
\end{proof}


\section{The equivariant $\tilde{S}$-polynomial}\label{s:tilde}

The $\tilde{S}$-polynomial of a lattice polytope was introduced by Borisov and Mavlyutov in \cite[Definition~5.3]{BMString}.
The goal of this section is to use the results of \cite{YoEquivariant} to introduce and study a representation-theoretic analogue of the $\tilde{S}$-polynomial. 

We continue with the notations of the previous section, and further assume that $C$ is the cone over a $d$-dimensional, $\Gamma$-invariant lattice polytope $P$, and $\rho: \Gamma \rightarrow \GL_{d + 1}(\Z)$ is an integer-valued representation. 
If $F$ is a face of $C$, then recall that $\Gamma_F$ denotes the stabilizer of $F$, with complex representation ring $R(\Gamma_F)$.  
For each 
face $F$ of $C$
and non-negative integer $m$, let $\chi_{F,m}$ denote the permutation representation of $\Gamma_F $ on the lattice points in $F \cap mP$.  Following \cite{YoEquivariant}, consider the power series of virtual representations $\phi_F[t]  = \sum_{i \ge 0} \phi_{F,i} t^i \in R(\Gamma_F)[[t]]$ defined by the equation  
\begin{equation}\label{e:Ehrhart}
\sum_{m \ge 0} \chi_{F,m} t^m = \frac{\phi_{F}[t]}{\det[I - \rho_F t]} \in R(\Gamma_F)[[t]],
\end{equation}
where $\det[I - \rho_F t]$ is defined by \eqref{e:determinant}. Observe that $\phi_F[t] = 1 \in R(\Gamma)$ when $F = \{ 0 \}$. 
For each non-zero face $F$, evaluating the characters of the terms of the above equation at $1 \in \Gamma_F$ yields
\[
\sum_{m \ge 0} f_{F \cap P}(m) t^m = \frac{ h^*_{F \cap P}(t)}{ (1 - t)^{\dim F} },
\] 
where $f_{F \cap P}(m)$ is the \define{Ehrhart polynomial} of $F \cap P$, with degree $\dim (F \cap P)$, and $h^*_{F \cap P}(t)$ is the \define{$h^*$-polynomial} of $F \cap P$, with degree at most $\dim (F \cap P)$ (see, for example, \cite{BRComputing}). In particular, if each $\phi_{F,i}$ is a representation, then $\phi_{F,i}(1)$ equals the dimension of the representation $\phi_{F,i}$, and hence
$\phi_F[t]$ is a polynomial of degree at most $d$. 

\begin{remark}\label{r:geometric1}
In subsequent sections, we will restrict attention to certain geometric cases. More specifically, we will assume that there exists a non-degenerate, $\Gamma$-invariant hypersurface with Newton polytope $P$. In this case, it is proved in \cite[Corollary~6.6]{YoRepresentations} that $\phi_{F,i}$ is a representation of $\Gamma_F$. 
\end{remark}

In fact, for each $\gamma \in \Gamma$, $\phi_F[t](\gamma)$ is a rational function in $t$ \cite[Lemma~6.3]{YoEquivariant}, and we have the following equivariant version of Ehrhart reciprocity. For each non-zero face $F$ of $C$ and positive integer $m$, let $\chi_{F,m}^\circ$ denote the permutation representation of $\Gamma_F $ on the lattice points in $\Int(F) \cap mP$, where $\Int(F)$ denotes the relative interior of $F$. Then Corollary~6.6 in \cite{YoEquivariant} states that
\begin{equation}\label{e:reciprocity}
\sum_{m \ge 1} \chi_{F,m}^\circ t^m = \frac{ t^{\dim F} \phi_{F}[t^{-1}] }{\det[I - \rho_F t]}  \in R(\Gamma_F)[[t]].
\end{equation}

\begin{corollary}\label{c:technical}
With the notation above, 
\[
\phi_F[t] =  \sum_{[F'] \in F/\Gamma_F}  \Ind^{\Gamma_F}_{\Gamma_{F'}}
\frac{ t^{\dim F'} \phi_{F'}[t^{-1}] \det[I - \rho_F t]}{\det[I - \rho_{F'} t]},
\]
where $F/\Gamma_F$ denotes the set of $\Gamma_F$-orbits of faces of $F$.
\end{corollary}
\begin{proof}
The result holds by definition when $F = \{ 0 \}$. Hence we may assume that $F$ is nonzero. 
For each $\gamma \in \Gamma_F$, we need to show that 
\[
\phi_F[t](\gamma) =  \sum_{  \substack{\gamma \cdot F' = F' \\ F' \subseteq F }}  
\frac{ t^{\dim F'} \phi_{F'}[t^{-1}](\gamma) \det(I - \rho_F(\gamma) t)}{\det(I - \rho_{F'}(\gamma) t)}.
\]
After dividing both sides by $\det(I - \rho_F(\gamma) t)$ and applying \eqref{e:reciprocity}, we need to show that 
\[
\sum_{m \ge 0} \chi_{F,m}(\gamma) t^m = 1 +  \sum_{  \substack{\gamma \cdot F' = F' \\ \{ 0 \} \ne F' \subseteq F }}  
\sum_{m \ge 1} \chi_{F',m}^\circ(\gamma) t^m .
\]
For each positive integer $m$, the coefficient of $t^m$ on both sides of the above equation equals the number of $\gamma$-fixed lattice points in $F \cap mP$. 
\end{proof}

We now introduce our representation-theoretic version of the $\tilde{S}$-polynomial of $P$, which restricts to the usual $\tilde{S}$-polynomial when $\Gamma$ is trivial. 
Recall that there is an inclusion reversing bijection between the faces $F$ of $C$ and the faces $F^*$ of 
 the dual cone $\check{C}$. 

\begin{definition}\label{d:stilde}
With the notation above,
\[
\tilde{S}_\Gamma(C,t) = \tilde{S}_{P, \Gamma}(C,t) = \sum_{[F] \in C/\Gamma} (-1)^{d + 1 - \dim F} \Ind^{\Gamma}_{\Gamma_F} \det(\rho_F)
\phi_F[t] G(F^*,t) \in R(\Gamma)[[t]],
\]
where $C/\Gamma$ denotes the set of $\Gamma$-orbits of faces of $C$, and $G(F^*,t)$ is defined by
 Definition~\ref{d:hg}.
\end{definition}

\begin{remark}\label{r:geometric2}
As in Remark~\ref{r:geometric1}, in subsequent sections we will assume that there exists a non-degenerate, $\Gamma$-invariant hypersurface with Newton polytope $P$. In this case, $\tilde{S}_\Gamma(t)$ is a polynomial of degree $d$, and its coefficients are representations (rather then virtual representations) of $\Gamma$ (see Remark~\ref{r:Stilde}). 
\end{remark}

\begin{example}
Suppose that $P$ is a simplex i.e. $P$ has precisely $d + 1$ vertices $\{ v_0, \ldots, v_d \} \subseteq C$.  
Then $\tilde{S}_\Gamma(C,t)$ has a concrete description as a graded permutation representation \cite[Corollary~8.1]{YoRepresentations}. More precisely, the coefficient of $t^m$ in $\tilde{S}_\Gamma(t)$ equals the permutation representation of $\Gamma$ acting on the lattice points $v$ in $\Int(mP)$ which can be written in the form $v = \sum_{i = 0}^{d} \alpha_i v_i$ for some $0 < \alpha_i < 1$.
\end{example}

When $\Gamma$ is trivial, the following lemma is proved in \cite[Remark~5.4]{BMString}.

\begin{lemma}\label{l:symmetry}
With the notation above, 
\[
\tilde{S}_\Gamma(C,t) = t^{d + 1}\tilde{S}_\Gamma(C,t^{-1}).  
\]
\end{lemma}
\begin{proof}
By Proposition~\ref{p:symmetry} and Lemma~\ref{l:ratio}, for any $\gamma \in \Gamma$, we compute 
\begin{align*}
 &t^{d + 1}\tilde{S}_\Gamma(C,t^{-1})(\gamma) =   \sum_{\gamma \cdot F = F} (-1)^{d + 1 - \dim F} 
 \det \rho_F(\gamma)
t^{\dim F}\phi_F[t^{-1}](\gamma) t^{\dim F^*}G(F^*,t^{-1})(\gamma)
  \\
  &= \sum_{\gamma \cdot F = F} (-1)^{d + 1 - \dim F}  \det \rho_F(\gamma)
t^{\dim F}\phi_F[t^{-1}](\gamma) \sum_{  \substack{\gamma \cdot F' = F' \\ F \subseteq F' }}  
\frac{ \det(tI - \rho_{F'}(\gamma)) G((F')^*,t)(\gamma) }{  \det(tI - \rho_{F}(\gamma))  } \\
&= \sum_{\gamma \cdot F' = F'} (-1)^{d + 1 - \dim F'}  \det \rho_{F'}(\gamma) G((F')^*,t)(\gamma)
\sum_{  \substack{\gamma \cdot F = F \\ F \subseteq F' }}  
\frac{ t^{\dim F}\phi_F[t^{-1}](\gamma)  \det(I - \rho_{F'}(\gamma)t)  }{  \det(I - \rho_{F}(\gamma)t)  }.
\end{align*}
The latter sum equals $\tilde{S}_\Gamma(C,t)(\gamma)$ by Corollary~\ref{c:technical}.  
\end{proof}


\section{Equivariant Hodge-Deligne polynomials of  hypersurfaces of tori}\label{s:formula}

In this section, we prove an explicit formula for the equivariant Hodge-Deligne polynomial of a $\Gamma$-invariant, non-degenerate hypersurface $X^\circ$ in a torus. When $\Gamma$ is trivial, this reduces to a reformulation of  Borisov and Mavlyutov \cite[proof of Proposition~5.5]{BMString} of a formula due to Batyrev and Borisov \cite[Theorem~3.24]{BBMirror}. In the case when $C$ is the cone over a simple polytope, a formula was given in \cite[Theorem~7.1]{YoRepresentations}. 

Let $\Gamma$ be a finite group acting algebraically on a  complex variety $Z$. Then the \define{equivariant Hodge-Deligne polynomial} $E_{\Gamma}(Z;u,v) \in R(\Gamma)[u,v]$ is a polynomial of virtual representations first considered in \cite[Section~5]{YoRepresentations}, 
and satisfying the following properties:
\begin{enumerate}
\item\label{e:prop1} If $Z$ is complete with at worst quotient singularities, then 
\[
E_{\Gamma}(Z;u,v) = \sum_{p,q} (-1)^{p + q} H^{p,q}(Z) u^p v^q,
\]
where $H^{p,q}(Z)$ is the $(p,q)^{\textrm{th}}$ piece of the complex cohomology of $Z$, regarded as a $\Gamma$-module.
\item If $U$ is a $\Gamma$-invariant, open subvariety of $Z$, then 
\[
E_{\Gamma}(Z) = E_{\Gamma}(U) + E_{\Gamma}( Z \setminus U).  
\]
\item If $\Gamma$ acts on a complex variety $Z'$, then 
\[
E_{\Gamma}(Z \times Z') =  E_{\Gamma}(Z)E_{\Gamma}(Z').
\]
\end{enumerate}

We refer the reader to Section~5 in \cite{YoRepresentations} for more details, including the definition of 
$E_{\Gamma}(Z;u,v)$ in terms of the action of $\Gamma$ on the mixed Hodge structure of the complex cohomology $H^*_c (Z)$ of $Z$ with compact support. We will need the following examples and facts.

\begin{example}\label{e:affine}
If $\Gamma$ acts algebraically on $\A^d$, then $E_\Gamma(\A^d) = H^d_c(\A^d)(uv)^d = (uv)^d$. 
\end{example}

\begin{example}\label{e:tori}\cite[Example~5.4]{YoRepresentations}
If $\Gamma$ acts linearly on a lattice $M$ of rank $d$ via a representation $\rho': \Gamma \rightarrow GL(M)$, then $\Gamma$ acts algebraically on the corresponding torus 
$T = \Spec \C[M]$, and, with the notation of \eqref{e:determinant}, 
\[
E_\Gamma(T;u,v) = \sum_{k = 0}^{d} (-1)^{d + k}  (\bigwedge^{d - k} \rho') \:  (uv)^k = \det[uvI - \rho'].
\]
\end{example}

\begin{proposition}\cite[Proposition 2.3]{LehRational}\label{p:induced}
Suppose a finite group $\Gamma$ acts a complex variety $Z$, and $Z$ admits a decomposition into locally closed subvarieties $Z = \coprod_{i \in I} Z_i$ which are permuted by $\Gamma$.  Then 
\[
E_\Gamma (Z) = \sum_{\iota \in I /\Gamma} \Ind_{\Gamma_i}^\Gamma E_{\Gamma_i}(Z_i),
\]
where $I/\Gamma$ denotes the set of orbits of $\Gamma$ acting on $I$, $i$ denotes a representative of the orbit $\iota$, and $\Gamma_i$ denotes the isotropy group of $i$ in $I$. 
In terms of characters, for any $\gamma$ in $\Gamma$, 
\[
E_\Gamma (Z)(\gamma) = \sum_{\gamma \cdot Z_i = Z_i} E_{\Gamma_i}(Z_i)(\gamma).
\]
\end{proposition}



We continue with the notations of the previous sections, and assume that $C$ is the cone over a $d$-dimensional, $\Gamma$-invariant lattice polytope $P$, and $\rho: \Gamma \rightarrow \GL_{d + 1}(\Z)$ is an integer-valued representation. 
 After possibly replacing $\Z^{d + 1}$ with a smaller lattice, we may assume that $\Z^{d + 1}$ is generated by lattice points in the affine span
$\aff(P)$ of $P$. If $M$ denotes a translate of $\aff(P) \cap \Z^{d + 1}$ to the origin, then we have an induced representation $\rho': \Gamma \rightarrow GL(M)$, such that the representation $\rho: \Gamma \rightarrow GL_{d + 1}(\R)$ is isomorphic, as a complex representation, to the direct sum of $\rho'$ and the trivial representation (see Section~2 in \cite{YoRepresentations} for details). Note that $\Gamma$ acts algebraically on the corresponding torus $T = \Spec \C[M]$. 
 If $u \in M$ corresponds to the monomial $\chi^u \in \C[M]$, then a hypersurface $X^\circ = \{ \sum_{ u \in P \cap M} a_u \chi^u = 0 \} \subseteq T$ defines a $\Gamma$-invariant hypersurface of $T$ if and only if $a_u = a_u' \in \C$ whenever $u$ and $u'$ lie in the same $\Gamma$-orbit of $P \cap M$.  The hypersurface
$X^\circ$ is  \define{non-degenerate} with respect to $P$ if $P$
 is the convex hull of $\{ u \in M \mid a_u \ne 0 \}$ in $M_\R$, and $\{ \sum_{ u \in Q \cap M} a_u \chi^u = 0 \}$ defines a smooth (possibly empty) hypersurface in $T$ for each face $Q$ of $P$. We refer the reader to Section~7 in \cite{YoEquivariant} for a discussion of the existence of non-degenerate, $\Gamma$-invariant hypersurfaces.

For the remainder of the section, $X^\circ$ will denote a $\Gamma$-invariant hypersurface of $T$ which is non-degenerate with respect to $P$. 
We define $E(C) = E(C;u,v)$ to be the expression
\begin{equation*}
(1/uv)[E_\Gamma(T)  + (-1)^{d + 1} \det(\rho) \sum_{[F] \in C/\Gamma}  u^{\dim F}
\Ind^{\Gamma}_{\Gamma_F} \det(\rho_F)
 \tilde{S}_{\Gamma_F}(F,u^{-1}v)G(C/F,uv)].
\end{equation*}
Here $C/\Gamma$ denotes the set of $\Gamma$-orbits of faces of $C$, and $C/F$ denotes the image of $C$ in the quotient of $\R^{d + 1}$ by the linear span of $F$. By Example~\ref{e:tori},
\[
E_\Gamma(T;u,v) = \det[uvI - \rho'] = \frac{\det[uvI - \rho ]}{uv - 1}.
\] 
Our goal is to prove that $E(C)$ equals the equivariant Hodge-Deligne polynomial $E_\Gamma(X^\circ)$ of $X^\circ$.

\begin{remark}\label{r:toric}
 The action of $\Gamma$ on the $\N$-graded, semi-group algebra $R = \C[C \cap \Z^{d + 1}]$ induces an action of $\Gamma$ on the  projective toric variety  $Y = \Proj R$ with torus $T$ via toric morphisms.
  The closure $X$ of $X^\circ$ in $Y$ is a \emph{non-degenerate}, $\Gamma$-invariant hypersurface in $Y$. 
  General philosophy about non-degenerate hypersurfaces, which was communicated to the author by Khovanski{\u\i}, suggests that the knowledge of 
 the equivariant Hodge-Deligne polynomials of all $\Gamma$-invariant, non-degenerate hypersurfaces of tori should be equivalent to the knowledge of the representations of $\Gamma$ on the intersection cohomology groups of all non-degenerate hypersurfaces of projective toric varieties. Since we will not need this correspondence, we leave this as an open problem, and refer the reader to \cite{BBMirror} for details in the case when $\Gamma$ is trivial. 
\end{remark}

In \cite{YoRepresentations}, the author proved an algorithm to compute  \[
E_\Gamma(X^\circ;u,v) = \sum_{0 \le p,q \le d - 1} e^{p,q}_\Gamma u^p v^q \in R(\Gamma)[u,v],\] extending an algorithm of Danilov and Khovanski{\u\i} in the case when $\Gamma$ is trivial \cite{DKAlgorithm}. Our proof will proceed by  verifying that $E(C)$ satisfies all the steps of the algorithm. 
The algorithm consists of three parts. It first determines  $e^{p,q}_\Gamma$ for $p + q > d - 1$,
 then determines the sums $\sum_q e^{p,q}_\Gamma$, and lastly determines $e^{p,q}_\Gamma$ for $p + q < d - 1$. 
We refer the reader to Section~6 in \cite{YoRepresentations} for details.


\define{Step~1} \cite[Section~6.1]{YoRepresentations}

The first part of the algorithm states that for $p + q > d - 1$, $e^{p,q}_\Gamma$ equals the coefficient of 
$u^p v^q$ in $\frac{E_\Gamma(T)}{uv}$.

Since the degree of $G(C/F,t)$ is bounded by $\frac{d + 1 - \dim F}{2}$ by definition, the total degree in 
$u$ and $v$ of the right hand side in the expression for $E(C)$  is bounded by
$d - 1$. Hence, for $p + q > d - 1$, the coefficient of $u^p v^q$ in  $E(C)$ equals the coefficient of 
$u^p v^q$ in $\frac{E_\Gamma(T)}{uv}$. 

\begin{remark}\label{r:Stilde}
If $p + q = d - 1$, then the coefficient of $u^p v^q$ in  the right hand side of the expression for $E(C)$ equals
$(-1)^{d + 1}$ times the coefficient of $t^{q + 1}$ in $\tilde{S}_\Gamma(C,t)$.  It follows from Theorem~\ref{t:formula} and the discussion in Section 6.1 in \cite{YoRepresentations} that the coefficient of $t^{q + 1}$ in $\tilde{S}_\Gamma(C,t)$ is equal to the representation of $\Gamma$ on the $(d - 1 - q,q)^{\textrm{th}}$ piece of the mixed Hodge structure on the primitive cohomology of the middle cohomology $H^{d - 1}_c (X^\circ)$ of $X^\circ$ with compact support. 
\end{remark}

\define{Step~2} \cite[Section~6.3]{YoRepresentations}

The second part of the algorithm states that for every $\gamma$ in $\Gamma$, 
\[
E_\Gamma(X^\circ;u,1) (\gamma) = (1/u)[E_\Gamma(T;u,1)(\gamma) + (-1)^{d + 1}  \det \rho(\gamma) \phi_{C}[u](\gamma)], 
\]
where $\phi_{C}[u]$ is defined by \eqref{e:Ehrhart}.

We want to compute the value of $E(C;u,1)$. That is, consider the expression
\begin{equation*}
(1/u)[ E_\Gamma(T;u,1)  + (-1)^{d + 1} \det(\rho) \sum_{[F] \in C/\Gamma}  u^{\dim F}
\Ind^{\Gamma}_{\Gamma_F}
\det(\rho_F) \tilde{S}_{\Gamma_F}(F,u^{-1})G(C/F,u)].
\end{equation*}

By Lemma~\ref{l:symmetry}, this simplifies to 
\[
(1/u)[E_\Gamma(T;u,1) + (-1)^{d + 1}  \det(\rho)  \sum_{[F] \in C/\Gamma}
\Ind^{\Gamma}_{\Gamma_F}
\det(\rho_F) \tilde{S}_{\Gamma_F}(F,u)G(C/F,u)].
\]
By the definition of $\tilde{S}_{\Gamma_F}(F,u)$, the evaluation of the character of 
\[
\Ind^{\Gamma}_{\Gamma_F}
\det(\rho_F) \tilde{S}_{\Gamma_F}(F,u)G(C/F,u)
\]
 at $\gamma \in \Gamma$ equals
\[ \sum_{\gamma \cdot F = F} \det \rho_F(\gamma)G(C/F,u)(\gamma)
\sum_{ \substack{ \gamma \cdot F' = F'  \\ F' \subseteq F         }  }  
(-1)^{\dim F - \dim F'} \det \rho_{F'}(\gamma) \phi_{F'}[u](\gamma)G((F')^*/F^*,u)(\gamma). 
\]
Rearranging, the latter sum is equal to 
\[
 \sum_{\gamma \cdot F' = F'} \det \rho_{F'}(\gamma)  \phi_{F'}[u](\gamma) 
 \sum_{ \substack{ \gamma \cdot F = F  \\ F' \subseteq F         }  } 
 (-1)^{\dim F - \dim F'} \det \rho_F(\gamma) G(C/F,u)(\gamma)G((F')^*/F^*,u)(\gamma). 
\]
After applying Lemma~\ref{l:convolution} to $C/F'$, this expression equals $\phi_{C}[u](\gamma)$. 
We conclude that 
\[
E(C;u,1)(\gamma) = (1/u)[E_\Gamma(T;u,1)(\gamma) + (-1)^{d + 1}  \det \rho(\gamma) \phi_{C}[u](\gamma)].
\]

\define{Step~3} \cite[Section~6.2]{YoRepresentations}

Let $C'$ be a $\Gamma$-invariant cone over a simple polytope which `refines' $C$ (Danilov and Khovanski{\u\i} use the term `majorizes' in \cite{DKAlgorithm}). That is, $C'$ satisfies the property that every ray is contained in precisely $d$ maximal faces, and there exists a $\Gamma$-equivariant function
\[
f: \{ \textrm{ non-zero faces of } C' \} \rightarrow \{ \textrm{ non-zero faces of } C \}, 
\]
such that:
\begin{enumerate}
\item For every ray $r'$ of $C'$, $f(r')$ is a ray of $C$ and the tangent cone of $C$ with respect to $f(r')$ is contained in the tangent cone of $C'$ with respect to $r'$. 
\item For every non-zero face $F'$ of $C'$, $f(F')$ is the cone 
generated by $\{ f(r') \mid r' \textrm{ is a ray of } F' \}$. 
\end{enumerate}
Geometrically, this corresponds to a projective, $\Gamma$-equivariant, partial resolution $Y' \rightarrow Y$ of the projective toric variety $Y$ determined by $P \subseteq C$.  The third part of the algorithm uses the fact that Poincar\'e duality holds for the closure of $X^\circ$ in $Y'$ in order to
compute $e^{p,q}_\Gamma$ for $p + q < d - 1$ using Step~1 together with induction on dimension. 
 Together with the two previous steps, this reduces the equality
$E_\Gamma(X^\circ) = E(C)$ to proving the following:

Fix $\gamma \in \Gamma$, and set 
\[
E(\bar{C'};u,v)(\gamma) 
=  \sum_{ \substack{ \gamma \cdot F' = F' \\ F' \ne \{ 0 \} }} E(f(F');u,v)(\gamma) 
\frac{ \det(uvI - \rho_{F'}(\gamma) ) }{ \det(uvI - \rho_{f(F')}(\gamma) ) }.
\]
We need to show that $E(\bar{C'};u,v)(\gamma) = (uv)^{d - 1}E(\bar{C'};u^{-1},v^{-1})(\gamma)$. 

\begin{remark}
With the notation above, $E(\bar{C'};u,v) \in R(\Gamma)[u,v]$ is the conjectural equivariant Hodge-Deligne polynomial of the closure of $X^\circ$ in $Y'$, and it remains to verify that the symmetry arising from Poincar\'e duality holds
(see \cite[Section~6.2]{YoRepresentations} for details).
\end{remark}


We will need the following three lemmas. 

\begin{lemma}\label{l:inverse}
\[
(uv)^{d -  1}E(C;u^{-1},v^{-1}) =  \sum_{  \substack{  [F] \in C/\Gamma \\ F \ne \{ 0 \} }} 
\Ind^{\Gamma}_{\Gamma_F} \frac{\det[I - \rho_C uv]}{\det[I - \rho_F uv]} E(F;u,v).
\]
\end{lemma}
\begin{proof}
We expand the left hand side of the above equation as
\begin{align*}
(uv)^{d} [E_\Gamma(T;u^{-1},v^{-1}) + &(-1)^{d + 1} \det(\rho_C) \times \\ 
&\sum_{[F] \in C/\Gamma}  u^{-\dim F}
\Ind^{\Gamma}_{\Gamma_F}
\det(\rho_F) \tilde{S}_{\Gamma_F}(F,uv^{-1})G(C/F,(uv)^{-1})].
\end{align*}
By Lemma~\ref{l:symmetry}, the latter sum equals
\begin{align*}
(uv)^{d}E_\Gamma(T;u^{-1},v^{-1}) + &(-1)^{d + 1} \det(\rho_C)  \times \\
&(uv)^{d}  \sum_{[F] \in C/\Gamma}  v^{-\dim F}
\Ind^{\Gamma}_{\Gamma_F}
\det(\rho_F) \tilde{S}_{\Gamma_F}(F,u^{-1}v)G(C/F,(uv)^{-1}).
\end{align*}
We first consider the term
\[
 \sum_{[F] \in C/\Gamma}  v^{-\dim F}
\Ind^{\Gamma}_{\Gamma_F}
\det(\rho_F) \tilde{S}_{\Gamma_F}(F,u^{-1}v)G(C/F,(uv)^{-1}).
\]
Fix $\gamma$ in $\Gamma$. By Proposition~\ref{p:symmetry}, the evaluation of the corresponding virtual characters at $\gamma$ equals
\[
\frac{1}{(uv)^{d + 1}} \sum_{\gamma \cdot F = F}  u^{\dim F} \det \rho_F(\gamma) \tilde{S}_{\Gamma_F}(F,u^{-1}v)(\gamma)
\sum_{ \substack{  \gamma \cdot F' = F' \\ F \subseteq F'   }  } 
\det(uvI - \rho_{C/F'}(\gamma)) G(F'/F,uv)(\gamma)
\]
\[
= \frac{1}{(uv)^{d + 1}} \sum_{\gamma \cdot F' = F'}  \det(uvI - \rho_{C/F'}(\gamma)) 
\sum_{ \substack{  \gamma \cdot F = F \\ F \subseteq F'   }  } 
u^{\dim F} \det \rho_F(\gamma) \tilde{S}_{\Gamma_F}(F,u^{-1}v)(\gamma)G(F'/F,uv)(\gamma).
\] 
Observing that for $F' \ne \{ 0 \}$, 
\[
\sum_{ \substack{  \gamma \cdot F = F \\ F \subseteq F'   }  } 
u^{\dim F} \det \rho_F(\gamma) \tilde{S}_{\Gamma_F}(F,u^{-1}v)(\gamma)G(F'/F,uv)(\gamma) = 
\]
\[
(-1)^{\dim F'}  \det \rho_{F'}(\gamma) [      uvE(F')(\gamma)            -  \frac{\det(uvI - \rho_{F'}(\gamma) )}{uv - 1}     ],
\]
the above expression simplifies to 
\[
\frac{1}{(uv)^{d + 1}}[   \det(uvI - \rho_{C}(\gamma))  +    \]\[     \sum_{ \substack{  \gamma \cdot F' = F' \\ F' \ne \{ 0 \}   }  }     (-1)^{\dim F'}  \det \rho_{F'}(\gamma)     \det(uvI - \rho_{C/F'}(\gamma))                  [      uvE(F')(\gamma)            -  \frac{\det(uvI - \rho_{F'}(\gamma) )}{uv - 1}     ]             ].
\]
Also, Lemma~\ref{l:once} implies that 
\[
(uv)^{d}E_\Gamma(T;u^{-1},v^{-1})(\gamma)  =  (uv)^{d} \frac{ \det ((uv)^{-1} I - \rho_C(\gamma))  }{(uv)^{-1} - 1} =  \frac{ (-1)^d \det  \rho_C(\gamma) \det (uv I - \rho_C(\gamma))  }{uv - 1} . 
\]

Putting this all together, we deduce that $(uv)^{d -  1}E(C;u^{-1},v^{-1})(\gamma)$ equals
\[
 \frac{ (-1)^d \det  \rho_C(\gamma) \det (uv I - \rho_C(\gamma))  }{uv - 1} + 
 \frac{(-1)^{d + 1} \det \rho_C(\gamma)}{uv}[   \det(uvI - \rho_{C}(\gamma))  +   
 \]
 \[
      \sum_{ \substack{  \gamma \cdot F' = F' \\ F' \ne \{ 0 \}   }  }      (-1)^{\dim F'}    \det \rho_{F'}(\gamma)    \det(uvI - \rho_{C/F'}(\gamma))                  [      uvE(F')(\gamma)            -  \frac{\det(uvI - \rho_{F'}(\gamma) )}{uv - 1}     ]    ].
\]
After rearranging, we obtain
\[
\sum_{ \substack{  \gamma \cdot F' = F' \\ F' \ne \{ 0 \}   }  }   \det(I - \rho_{C/F'}(\gamma)uv)   E(F')(\gamma)  + 
 \frac{ \det (I - \rho_C(\gamma)uv)  }{ uv(1 - uv)}  \sum_{  \gamma \cdot F' = F'  } (-1)^{\dim F'} \det \rho_{F'}(\gamma).  
\]
By Lemma~\ref{l:mobius}, the second term in the above expression is zero, and we conclude that
\[
(uv)^{d -  1}E(C;u^{-1},v^{-1})(\gamma) = \sum_{ \substack{  \gamma \cdot F' = F' \\ F' \ne \{ 0 \}   }  }   \det(I - \rho_{C/F'}(\gamma)uv)   E(F')(\gamma).
\]
\end{proof}

\excise{
We start with an expression
\[
E (\Sigma; u,v) = \sum_{[\tau'] \in \Sigma /\Gamma} \Ind_{\Gamma_{\tau'}}^\Gamma [E(f(\tau'))E_{\Gamma_{\tau'}}( T_{\tau',f})],
\]
where $T_{\tau',f} := \Spec \C[M_{\tau'}/M_{\tau}]$. We need to show that $E (\Sigma; u,v) = (uv)^{d - 1} E (\Sigma; u^{-1},v^{-1})$. Here $f(\tau')$ is a face of $C$. 
Fix $\gamma \in \Gamma$. We compute  

\[
(uv)^{d - 1} E (\Sigma; u^{-1},v^{-1})(\gamma) = (uv)^{d - 1} \sum_{\gamma \cdot \tau' = \tau'} 
E(f(\tau'); u^{-1},v^{-1})(\gamma)E_{\Gamma_{\tau'}}( T_{\tau',f}; u^{-1},v^{-1})(\gamma). 
\]
By Lemma~\ref{l:inverse}, the latter sum is equal to 
\[
 (uv)^{d - 1} \sum_{\gamma \cdot \tau' = \tau'} 
E(f(\tau'); u^{-1},v^{-1})(\gamma)E_{\Gamma_{\tau'}}( T_{\tau',f}; u^{-1},v^{-1})(\gamma). 
\]
}

The following lemma easily follows from the proof of Lemma~\ref{l:mobius}, and the corresponding statement when $\Gamma$ is trivial. 

\begin{lemma}\label{l:mobius2}
 Let $C$ denote the cone over a $\Gamma$-invariant polyhedron $R$, and consider the poset of 
 cones $F_Q$ over $\gamma$-invariant faces $Q$ of $R$, for some $\gamma \in \Gamma$. 
Then the associated M\"obius function satisfies 
\[
\mu_{\gamma,R}(\{ 0 \}, F_{Q}) = 
 \left\{\begin{array}{cl}
(-1)^{\dim F_{Q}} \det \rho_{F_{Q}} (\gamma) & \text{if } Q \textrm{ is bounded }  \\ 0 & \text{otherwise}. \end{array}\right. 
\]
\end{lemma}

\begin{lemma}\label{l:polyhedron}
 Let $C$ denote the cone over a simple, $\Gamma$-invariant polyhedron $R$, and consider the poset of 
 cones $F_Q$ over $\gamma$-invariant faces $Q$ of $R$, for some $\gamma \in \Gamma$. Then
 \[
 t^{d + 1} \sum_{F_Q \ne \{ 0 \} } \det( t^{-1}I - \rho_{F_Q}(\gamma)) = 
 - \sum_{  \substack{ F_Q \ne \{ 0 \} \\   Q \textrm{ bounded }  } } \det( tI - \rho_{F_Q}(\gamma)). 
 \]
\end{lemma}
\begin{proof}
Let $\mathcal{P}$ denote the poset of  cones over $\gamma$-invariant faces of $R$, and consider the function $h: \mathcal{P} \rightarrow \Z[t]$ defined by $h( \{ 0 \}) = 0$ and $h(F_Q) =   \det( t^{-1}I - \rho_{F_Q}(\gamma)) $ if $F_Q \ne \{ 0 \}$. By Lemma~\ref{l:mobius2}, M\"obius inversion implies that 
\[
h( \{ 0 \}) = 0 = \sum_{F_Q \ne \{ 0 \} } \det( t^{-1}I - \rho_{F_Q}(\gamma)) + 
 \sum_{  \substack{ F_Q \ne \{ 0 \} \\   Q \textrm{ bounded }  } } (-1)^{\dim F_{Q}} \det \rho_{F_{Q}} (\gamma)
 g(Q),
\]
where $g(Q) = \sum_{Q \subseteq Q'} \det( t^{-1}I - \rho_{F_{Q'}}(\gamma)) $. In order to compute $g(Q)$,  observe that since $R$ is simple, $Q$ is contained in precisely $\codim Q$ facets of $R$. 
Moreover, $\rho_{C/F_Q}(\gamma)$ is conjugate to the permutation matrix associated to the action of $\gamma$ on these facets.  Let $\{ V_1, \ldots, V_s \}$ denote the $\gamma$-orbits of facets containing $Q$. 
For any (possibly empty) subset $I \subseteq \{ 1, \ldots , s \}$, let $Q_I$  be the
intersection of the facets $\{ F_j \in V_i \mid i \in I \}$. 
Then the faces $\{ Q_I \mid I \subseteq \{ 1, \ldots , s \} \}$ are precisely the faces of $R$ which contain $Q$ and are fixed by $\gamma$.
We compute
\begin{align*}
g(Q) &=   \det( t^{-1}I - \rho_{F_Q}(\gamma)) \sum_{I \subseteq \{ 1, \ldots , s \}}  \det( t^{-1}I - 
\rho_{F_{Q_I}/F_Q}(\gamma))  \\
&=  \det( t^{-1}I - \rho_{F_Q}(\gamma)) \sum_{I \subseteq \{ 1, \ldots , s \}} \prod_{i \in I} (t^{-|V_i|} - 1) \\
&=  \det( t^{-1}I - \rho_{F_Q}(\gamma)) \prod_{i  = 1}^s (  (t^{-|V_i|} - 1) + 1) \\
&= t^{-\codim Q} \det( t^{-1}I - \rho_{F_Q}(\gamma)). 
\end{align*}
By Lemma~\ref{l:once},  $t^{d + 1}g(Q) = (-1)^{\dim F_{Q}} \det \rho_{F_{Q}} (\gamma) \det( tI - \rho_{F_Q}(\gamma))$, and the result follows.
\end{proof}

We are now ready to prove that $E(\bar{C'};u,v)(\gamma) = (uv)^{d - 1}E(\bar{C'};u^{-1},v^{-1})(\gamma)$. 
We compute 
\[
(uv)^{d - 1}E(\bar{C'};u^{-1},v^{-1})(\gamma) =   (uv)^{d - 1} \sum_{ \substack{ \gamma \cdot F' = F' \\ F' \ne \{ 0 \} }} E(f(F');u^{-1},v^{-1})(\gamma) 
\frac{ \det((uv)^{-1}I - \rho_{F'}(\gamma) ) }{ \det((uv)^{-1}I - \rho_{f(F')}(\gamma) ) }.
\]
By Lemma~\ref{l:inverse}, the latter sum is equal to 
\[
 \sum_{ \substack{ \gamma \cdot F' = F' \\ F' \ne \{ 0 \} }} (uv)^{d + 1 - \dim f(F')}  \frac{ \det((uv)^{-1}I - \rho_{F'}(\gamma) ) }{ \det((uv)^{-1}I - \rho_{f(F')}(\gamma) ) } \sum_{ \substack{ \gamma \cdot F = F \\ \{ 0 \} \ne F \subseteq f(F') }}  \frac{ \det(I - \rho_{f(F')}(\gamma)uv ) }{ \det(I - \rho_{F}(\gamma)uv ) } E(F)(\gamma).
\]
\[
= \sum_{ \substack{ \gamma \cdot F = F \\ F \ne \{ 0 \} }}  \frac{ E(F)(\gamma) }{ \det(I - \rho_{F}(\gamma)uv ) }
 \sum_{ \substack{ \gamma \cdot F' = F' \\ F \subseteq f(F') }} (uv)^{d + 1} \det((uv)^{-1}I - \rho_{F'}(\gamma) ).
\]

Note that every maximal face  $F'$ in $C'$ corresponds to a half plane. If we fix a non-empty $\gamma$-invariant face $F$ of $C$ and let $S'$ be the intersection of the half planes corresponding to maximal faces $F'$ of $C'$ satisfying $F \subseteq f(F')$, then $S' = S \times \R^{\dim F - 1}$, where $S$ is the cone over a $\gamma$-invariant polyhedron $R$ with a non-empty bounded face, and 
  the representation of $\Gamma$ on  $\R^{\dim F - 1}$ is identified with $\rho_F'$, the representation associated with the action of $\Gamma$ on the affine span of the face of $P$   corresponding to $F$. 
 The non-empty faces $R_{F'}$ 
of $R$ are in bijective correspondence with the non-empty faces $F'$ of $C'$ such that $F \subseteq f(F')$,  and satisfy $\dim R_{F'} = \dim F' - \dim F$. Moreover, the non-empty bounded faces of $R$ are in bijective correspondence with the non-zero faces $F'$ of $C'$ such that $f(F') = F$. Hence Lemma~\ref{l:polyhedron} implies that
\[
 \sum_{ \substack{ \gamma \cdot F' = F' \\ F \subseteq f(F') }}
 \frac{ (uv)^{d + 2 - \dim F} \det((uv)^{-1}I - \rho_{F'}(\gamma) )( (uv)^{-1} - 1) }{ \det((uv)^{-1}I - \rho_{F}(\gamma) ) } \]
 \[ = -  \sum_{ \substack{ \gamma \cdot F' = F' \\ F = f(F') }}
  \frac{ \det(uvI - \rho_{F'}(\gamma) )( uv - 1) }{ \det(uvI - \rho_{F}(\gamma) ) }.
 \]
 By Lemma~\ref{l:once}, this simplifies to 
\[
 \sum_{ \substack{ \gamma \cdot F' = F' \\ F \subseteq f(F') }}
  (uv)^{d + 1} \det((uv)^{-1}I - \rho_{F'}(\gamma) )
 =  (-1)^{\dim F} \det \rho_F(\gamma) \sum_{ \substack{ \gamma \cdot F' = F' \\ F = f(F') }}
  \det(uvI - \rho_{F'}(\gamma) ).
 \]
 Combining this with our previous expression yields
 \begin{align*}
 (uv)^{d - 1}E(\bar{C'};u^{-1},v^{-1})(\gamma) &=   \sum_{ \substack{ \gamma \cdot F = F \\ F \ne \{ 0 \} }}  \frac{ (-1)^{\dim F} \det \rho_F(\gamma) E(F)(\gamma) }{ \det(I - \rho_{F}(\gamma)uv ) }
 \sum_{ \substack{ \gamma \cdot F' = F' \\ F = f(F') }} \det(uvI - \rho_{F'}(\gamma) ) \\
 &=  \sum_{ \substack{ \gamma \cdot F' = F' \\ F' \ne \{ 0 \} }} E(f(F'))(\gamma) 
\frac{ \det(uvI - \rho_{F'}(\gamma) ) }{ \det(uvI - \rho_{f(F')}(\gamma) ) } \\
&= E(\bar{C'};u,v)(\gamma).
 \end{align*}

In conclusion, we have proven the following formula for the equivariant Hodge-Deligne polynomial of a $\Gamma$-invariant, non-degenerate hypersurface in a torus.

\begin{theorem}\label{t:formula}
Let $\Gamma$ be a finite group, and let 
$\rho: \Gamma \rightarrow \GL_{d + 1}(\Z)$ be an integer-valued representation. Let $C \subseteq \R^{d + 1}$ be a $(d + 1)$-dimensional, $\Gamma$-invariant cone over a lattice polytope $P$. 
If $X^\circ$ is a $\Gamma$-invariant hypersurface of the corresponding torus $T$ which is non-degenerate with respect to $P$, then $E_\Gamma(X^\circ;u,v)$ equals
\[
(1/uv)[E_\Gamma(T)  + (-1)^{d + 1} \det(\rho) \sum_{[F] \in C/\Gamma}  u^{\dim F}
\Ind^{\Gamma}_{\Gamma_F} \det(\rho_F)
 \tilde{S}_{\Gamma_F}(F,u^{-1}v)G(C/F,uv)],
\]
where $C/\Gamma$ denotes the set of $\Gamma$-orbits of faces of $C$, $C/F$ denotes the image of $C$ in the quotient of $\R^{d + 1}$ by the linear span of $F$, and
$E_\Gamma(T;u,v) = \frac{\det[uvI - \rho ]}{uv - 1}$.
\end{theorem}
\begin{proof}
This follows from the main result in \cite{YoRepresentations}, which is an algorithm for computing 
$E_\Gamma(X^\circ;u,v)$, together with the above computations, which show that the conjectured formula for $E_\Gamma(X^\circ;u,v)$ satisfies all the properties which uniquely determine $E_\Gamma(X^\circ;u,v)$.
\end{proof}


\section{Equivariant stringy invariants of hypersurfaces}\label{s:stringy}

The goal of this section is to explicitly compute the equivariant stringy invariant of a non-degenerate, $\Gamma$-invariant hypersurface in a Gorenstein, projective toric variety. We refer the reader to \cite{FulIntroduction} and \cite{TevCompactifications} for details on toric varieties and non-degenerate hypersurfaces. 

We continue with the notations of the previous sections. That is,  $\Gamma$ is a finite group
with a representation
$\rho: \Gamma \rightarrow \GL_{d + 1}(\Z)$, and
 $C$ is the cone over a $d$-dimensional, $\Gamma$-invariant lattice polytope $P$. 
We may and will assume that   $\Z^{d + 1}$ is generated by lattice points in the affine span
$\aff(P)$ of $P$. If $M$ denotes a translate of $\aff(P) \cap \Z^{d + 1}$ to the origin, then we have an induced representation $\rho': \Gamma \rightarrow GL(M)$. We have an induced action of $\Gamma$ on 
the  projective toric variety  $Y = Y_P = \Proj \C[C \cap \Z^{d + 1}]$ with torus $T = \Spec \C[M]$ via toric morphisms.
Let $X$ be a $\Gamma$-invariant, non-degenerate hypersurface in $Y$ with respect to $P$ (see Remark~\ref{r:toric}). The toric variety $Y$ has a stratification $Y = \cup_F T_F$ into torus orbits $T_F$ indexed by the non-zero faces $F$ of $C$, which induces a stratification $X = \cup_F X_F^\circ$, where 
$X_F^\circ$ is a non-degenerate hypersurface in the $(\dim F - 1)$-dimensional torus $T_F$. We let $N = \Hom(M,\Z)$ with its induced $\Gamma$-action,  and let $\Sigma$ denote the normal fan to $P$. We assume throughout this section that $Y$ is \define{Gorenstein}. This assumption is equivalent to assuming that there exists a piecewise-linear function $\psi_{\Sigma}: N_\R \rightarrow \R$ with respect to $\Sigma$ such that $\psi_{\Sigma}(v) = 1$ for all primitive integer vectors $v$ on the rays of $\Sigma$. 

We first define the equivariant stringy invariant of a complex, Gorenstein variety $Z$ with an action of $\Gamma$ and at worst canonical singularities. Assume there exists a $\Gamma$-equivariant resolution of singularities $\pi: Z' \rightarrow Z$ such that the relative canonical divisor $K_{Z'/Z} = \sum_{i = 1}^r a_i D_i$ is a simple normal crossings divisor. 
The assumption that $Z$ is Gorenstein means that $Z$ admits a canonical divisor, and the assumption that $Z$ has at worst canonical singularities means that the coefficients $a_i$ are non-negative integers. 
The divisor $K_{Z'/Z}$ is supported on the exceptional locus of  $\pi$ and is $\Gamma$-invariant. Note that $\Gamma$ permutes the prime divisors $\{ D_i \mid 1 \le i \le r \}$, and hence acts on the set 
$\{ 1, \ldots, r \}$. 
For each (possibly empty) subset $J$ of $\{ 1, \ldots, r \}$, let $D_J^\circ = \cap_{j \in J} D_j \setminus 
\cup_{i \notin J} D_i$. For example, $D_J^\circ = Z' \setminus \cup_{i = 1}^r D_i$ when $J = \emptyset$. 
Then $Z'$ admits a stratification as a disjoint union of the locally closed subvarieties $D_J^\circ$. 

\begin{definition}\label{d:esi}
With the notation above, the \define{equivariant stringy invariant} $E_{\st, \Gamma}(Z;u,v) \in R(\Gamma)[[u,v]]$ of $Z$ is the power series of virtual representations defined as follows:
for each $\gamma$ in $\Gamma$, 
\[
E_{\st, \Gamma}(Z;u,v)(\gamma) = \sum_{  \substack{ J \subseteq \{ 1, \ldots , r \} \\ \gamma \cdot J = J } }
E_{\Gamma_J} (D_J^\circ;u,v)(\gamma) \prod_{j = 1}^l \frac{(uv)^{|J_j|}  - 1 }{  (uv)^{|J_j|(a_j + 1)}  - 1 },
\]
where $J_1, \ldots, J_l$ denote the $\gamma$-orbits of $J$, $a_j$ denotes the coefficient in $K_{Z'/Z}$ of a prime divisor in $J_j$, and $\Gamma_J$ denotes the stabilizer of $J$. 
\end{definition}

\begin{remark}
When $\gamma = 1$, $E_{\st, \Gamma}(Z)(\gamma)$ equals the 
stringy invariant of $Z$
introduced by Batyrev in 
\cite{BatStringy}. 
\end{remark}

\begin{remark}\label{r:crepant}
Observe that if $\pi: Z' \rightarrow Z$ is a crepant resolution, i.e. $K_{Z'/Z} = 0$, then  $E_{\st, \Gamma}(Z) = E_{\Gamma}(Z')$ by Proposition~\ref{p:induced}. 
\end{remark}

\begin{remark}\label{r:motivic}
Using Example~\ref{e:affine}, one can generalize motivic integration to complex varieties with a $\Gamma$-action, and define the equivariant stringy invariant as a motivic integral (cf. Section~3.6 and Definition~7.7 in \cite{VeyArc} when $\Gamma$ is trivial). In fact, this was the motivation for the form of Definition~\ref{d:esi}. Using this approach, one should be able to show that $E_{\st, \Gamma}(Z)$ is independent of the choice of equivariant resolution. In our case of interest, we will see this independence later by other means. Since we will not use this remark, we leave the details as an open problem. 
\end{remark}

Let us now return to our situation. A result of Abramovich and Wang \cite{AWEquivariant} implies that there exists a  $\Gamma$-invariant, smooth fan $\Sigma'$   refining
the normal fan $\Sigma$ to $P$. Since $Y$ is Gorenstein by assumption,  the associated $\Gamma$-equivariant toric resolution of singularities $Y' = Y'(\Sigma') \rightarrow Y = Y(\Sigma)$ admits a torus-invariant relative canonical divisor 
\[
K_{Y'/Y} = \sum_{v_i'} (\psi_\Sigma(v_i') - 1)E_i,
\] 
where $v_i'$ varies over the primitive integer vectors of the rays of $\Sigma'$, and $E_i$ denotes the torus-invariant prime divisor in $Y'$ corresponding to $v_i'$. 
If $X'$ denotes the closure of $X^\circ = X \cap T$ in $Y'$, then we have an induced $\Gamma$-equivariant resolution of singularities $\pi: X' \rightarrow X$ with $\Gamma$-invariant, simple normal crossings,  relative canonical divisor $K_{X'/X} = K_{Y'/Y}|_{X'}$ (see, for example, \cite[Theorem~1.4]{TevCompactifications}). We will call such a resolution a $\Gamma$-equivariant \define{toric resolution of singularities}. If $T_{\tau'} \subseteq Y'$ denotes the torus orbit corresponding to a cone $\tau'$ in $\Sigma'$, then the value of the  equivariant stringy invariant of $X$ at $\gamma$ in $\Gamma$ equals
\[
E_{\st, \Gamma}(X)(\gamma) = \sum_{  \substack{ \tau' \in \Sigma' \\ \gamma \cdot \tau' = \tau' } }
E_{\Gamma_{\tau'}} (X' \cap T_{\tau'})(\gamma) \prod_{j = 1}^l \frac{(uv)^{|J_j|}  - 1 }{  (uv)^{|J_j|\psi_\Sigma(v_j')}  - 1 },
\]
where $J_1, \ldots, J_l$ denote the $\gamma$-orbits of rays of $\tau'$, $v_j'$ denotes a primitive integer vector of a ray in  $J_j$, and $\Gamma_{\tau'}$ denotes the stabilizer of $\tau'$. For a fixed $\tau' \in \Sigma'$, let $\tau$ denote the smallest cone in   $\Sigma$ containing $\tau'$, with corresponding torus orbit $T_{\tau} \subseteq Y$. 
As in Section~4 and Section~6.2 in \cite{YoRepresentations}, $\pi: X' \rightarrow X$ induces a $\Gamma$-equivariant 
projection 
\[
X' \cap T_{\tau'} \cong (X \cap T_{\tau}) \times T_{\tau',\pi} \rightarrow X \cap T_{\tau},
\]
for a torus $T_{\tau, \pi}$, inducing an equality
\[
E_{\Gamma_{\tau'} }(X' \cap T_{\tau'})(\gamma) = E_{\Gamma_\tau} (X \cap T_{\tau})(\gamma) \frac{\det( uvI - \rho_{\tau}(\gamma)  )}{ \det( uvI - \rho_{\tau'}(\gamma)  )  },
\]
where  $\rho_{\tau}$ (respectively $\rho_{\tau'}$) denotes the representation of $\Gamma$ on the linear span of $\tau$ (respectively $\tau'$). Since the restriction of $\rho_{\tau'}$ to the cyclic group $\langle \gamma\rangle$ generated by $\gamma$ is isomorphic to the permutation representation of $\langle \gamma\rangle$ acting on the rays of $\tau'$, it follows that
\begin{align*}
E_{\st, \Gamma}(X)(\gamma) &=  \sum_{  \substack{ \tau' \in \Sigma' \\ \gamma \cdot \tau' = \tau' } }
E_{\Gamma_\tau}(X \cap T_{\tau})(\gamma) \frac{\det( uvI - \rho_{\tau}(\gamma)  )}{ \det( uvI - \rho_{\tau'}(\gamma)  )  } \prod_{j = 1}^l \frac{(uv)^{|J_j|}  - 1 }{  (uv)^{|J_j|\psi_\Sigma(v_j')}  - 1 } \\
&=  \sum_{  \substack{ \tau' \in \Sigma' \\ \gamma \cdot \tau' = \tau' } }
E_{\Gamma_\tau} (X \cap T_{\tau})(\gamma) \det( uvI - \rho_{\tau}(\gamma)  ) \prod_{j = 1}^l \frac{1 }{  (uv)^{|J_j|\psi_\Sigma(v_j')}  - 1 }.
\end{align*}
where $J_1, \ldots, J_l$ denote the $\gamma$-orbits of rays of $\tau'$, and $v_j'$ denotes a primitive integer vector of a ray in  $J_j$. On the other hand, since $\tau'$ is a unimodular cone, 
\[
\sum_{ \substack{ v \in \Int(\tau') \cap N \\ \gamma \cdot v = v  }} (uv)^{- \psi_{\Sigma}(v)} = \prod_{j = 1}^l \frac{  (uv)^{- |J_j|\psi_\Sigma(v_j')} }{   1 - (uv)^{-|J_j|\psi_\Sigma(v_j')}  } = \prod_{j = 1}^l \frac{1 }{  (uv)^{|J_j|\psi_\Sigma(v_j')}  - 1 }. 
\]
Putting these two expressions together yields the equality,
\begin{equation}\label{e:temp}
E_{\st, \Gamma}(X)(\gamma) =  \sum_{ \substack{ \tau \in \Sigma \\ \gamma \cdot \tau = \tau } } E_{\Gamma_\tau} (X \cap T_{\tau})(\gamma) \det( uvI - \rho_{\tau}(\gamma)  ) \sum_{ \substack{ v \in \Int(\tau) \cap N \\ \gamma \cdot v = v  }} (uv)^{- \psi_{\Sigma}(v)},
\end{equation}
which shows that $E_{\st, \Gamma}(X)(\gamma)$ is independent of the choice of $\Sigma'$. 
If $\tau$ is non-zero, then the assumption that $Y$ is Gorenstein implies that $\tau$ is the cone over the lattice polytope $Q_\tau = \tau \cap \psi_{\Sigma}^{-1}(1)$, and $mQ_\tau = \tau \cap \psi_{\Sigma}^{-1}(m)$ for every non-negative integer $m$. Following \eqref{e:Ehrhart}, we may consider the expression
\[
\sum_{m \ge 0} \chi_{\tau,m} t^m = \frac{\phi_{\tau}[t]}{\det[I - \rho_\tau t]} \in R(\Gamma_\tau)[[t]],
\]
where $\chi_{\tau,m} $ denotes the permutation representation of $\Gamma_\tau$ on the lattice points in $\tau \cap \psi_{\Sigma}^{-1}(m)$, and
$\phi_{\tau}[t] \in R(\Gamma_\tau)[[t]]$. Observe that when $\tau = \{ 0 \}$, $\phi_{\tau}[t] = \det[I - \rho_\tau t] = 1 \in R(\Gamma)$, and hence the above equation holds in this case. Evaluating \eqref{e:reciprocity} at $\gamma \in \Gamma$ yields
\[
 \sum_{ \substack{ v \in \Int(\tau) \cap N \\ \gamma \cdot v = v  }} t^{\psi_{\Sigma}(v)} = 
 \frac{ t^{\dim \tau} \phi_{\tau}[t^{-1}](\gamma) }{\det(I - \rho_\tau(\gamma) t)} 
\]
Setting $t = (uv)^{-1}$ and substituting into \eqref{e:temp} yields
\[
E_{\st, \Gamma}(X)(\gamma) = \sum_{ \substack{ \tau \in \Sigma \\ \gamma \cdot \tau = \tau } } E_{\Gamma_\tau} (X \cap T_{\tau})(\gamma) \phi_{\tau}[uv](\gamma).
\]
Finally, recall that there is a bijection between the cones $\tau$ in $\Sigma$ and the non-zero cones of $C$, such that, with our previous notations, $T_F = T_\tau$ and $\Gamma_F = \Gamma_\tau$. 
We summarize the results of our discussion in the following proposition. 

\begin{proposition}\label{p:stringy}
Let $\Gamma$ be a finite group, and let 
$\rho: \Gamma \rightarrow \GL_{d + 1}(\Z)$ be an integer-valued representation. Let $C \subseteq \R^{d + 1}$ be a $(d + 1)$-dimensional, $\Gamma$-invariant cone over a lattice polytope $P$. Let $\Sigma$ denote the normal fan of $P$ with corresponding toric variety $Y = Y(\Sigma)$. If $X$ is a $\Gamma$-invariant hypersurface of $Y$ which is non-degenerate with respect to $P$, then the equivariant stringy invariant of $X$ is given by  
\[
E_{\st, \Gamma}(X; u,v) = \sum_{ \substack{ [F] \in C/\Gamma \\ F \ne \{ 0 \} } } 
 \Ind^{\Gamma}_{\Gamma_F}  E_{\Gamma_F} (X \cap T_{F} ; u,v) \phi_{\tau}[uv],
\]
where $C/\Gamma$ denotes the set of $\Gamma$-orbits of faces of $C$, and $\tau$ is the cone in $\Sigma$ corresponding to the face $F$ of $C$. Moreover, an explicit formula for  
$E_{\Gamma_F} (X \cap T_{F} ; u,v)$ is given by Theorem~\ref{t:formula}. 
If $X$ admits a crepant, $\Gamma$-equivariant, toric resolution $X' \rightarrow X$, then 
$E_{\st, \Gamma}(X) = E_\Gamma (X')$. 
\end{proposition}

\begin{remark}
Using Lemma~\ref{l:inverse}, together with \eqref{e:temp} and the above proposition, one verifies that 
$E_{\st, \Gamma}(X; u,v) = (uv)^{d - 1}E_{\st, \Gamma}(X; u^{-1},v^{-1})$. 
\end{remark}

The polytope $P$ is \define{reflexive} if it contains a unique interior lattice point $v$, and, after setting 
$M = (\aff(P) \cap \Z^{d + 1}) - v$, where $\aff(P)$ is the affine span of $P$, and regarding $P$ as a polytope in $M$ after translation, 
 every non-zero lattice point in $M$ lies on the boundary of $mP$ for some positive integer $m$.  
 We refer the reader to Section~1 in \cite{BNCombinatorial} for a thorough discussion of reflexive polytopes. If $P$ is reflexive, then the dual cone $\hat{C}$ of $C$ is the cone over a lattice polytope $P^*$, which itself is reflexive. With the notation of \eqref{e:Ehrhart}, Corollary~6.9 in \cite{YoEquivariant} states that $P$ is reflexive if and only if 
\begin{equation}\label{e:reflex}
\phi_{C}[t] = t^d\phi_C[t^{-1}]. 
\end{equation}


If $Y^*$ denotes the projective toric variety corresponding to the normal fan of $P^*$, then the toric varieties $Y$ and $Y^*$ are Fano i.e. their anti-canonical divisors are ample, and, in particular, are Gorenstein. 
When $\Gamma$ is trivial, the formula below is a reformulation due to Borisov and Maylutov \cite[Theorem~7.2]{BMString} of a formula of Batyrev and Borisov \cite[Theorem~4.14]{BBMirror}. 

\begin{corollary}\label{c:reflexive}
Let $\Gamma$ be a finite group, and let 
$\rho: \Gamma \rightarrow \GL_{d + 1}(\Z)$ be an integer-valued representation. Let $C \subseteq \R^{d + 1}$ be a $(d + 1)$-dimensional, $\Gamma$-invariant cone over a reflexive polytope $P$. If $X$ is a $\Gamma$-invariant hypersurface of $Y$ which is non-degenerate with respect to $P$, then the equivariant stringy invariant of $X$ is given by  
\begin{equation}\label{e:reflexive}
E_{\st, \Gamma}(X; u,v) = \frac{\det(\rho)}{uv} \sum_{ [F] \in C/\Gamma } (-u)^{\dim F}  \Ind^{\Gamma}_{\Gamma_F}  \det (\rho_F) \tilde{S}_{\Gamma_F}(F, u^{-1}v) \tilde{S}_{\Gamma_F}(F^*, uv), 
\end{equation}
where $C/\Gamma$ denotes the set of $\Gamma$-orbits of faces of $C$.
\end{corollary}
\begin{proof}
Let $v^*$ be the unique interior lattice point in $P^*$. If $N = \Hom (M, \Z)$, then the dual lattice admits  a 
$\Gamma$-equivariant isomorphism $(\Z^{d + 1})^* \cong N \oplus (\Z \cdot v^*)$, and projection onto the first co-ordinate gives $\Gamma$-equivariant isomorphisms between the semi-groups
$F^* \cap (\Z^{d + 1})^*$ and $\tau \cap N$, where $F^*$ runs over the faces of the dual cone $\hat{C}$, 
and $\tau$ runs over the cones of the normal fan to $P$. 
Hence, Proposition~\ref{p:stringy} and Theorem~\ref{t:formula} imply that for every $\gamma \in \Gamma$, 
$E_{\st, \Gamma}(X)(\gamma)$ is equal to  
\begin{align*}
(1/uv) &\sum_{ \substack{ 0 \ne F \subseteq C \\ \gamma \cdot F = F } } \phi_{F^*}[uv](\gamma)
[  E_{\Gamma_F}(T_F)(\gamma) \\
 &+ (-1)^{\dim F}    \det \rho_F(\gamma)      
\sum_{ \substack{ \tilde{F} \subseteq F  \\ \gamma \cdot \tilde{F}  = \tilde{F}  }  } u^{\dim \tilde{F} } \det \rho_{\tilde{F} }(\gamma)  \tilde{S}_{\Gamma_{\tilde{F}}}(\tilde{F} ,u^{-1}v)(\gamma)  G(F/\tilde{F} ,uv)(\gamma)  ].
\end{align*}
 After rearranging, this expression becomes
\begin{align*}
&(1/uv) \sum_{ \substack{ 0 \ne F \subseteq C \\ \gamma \cdot F = F } } \phi_{F^*}[uv](\gamma)
[  E_{\Gamma_F}(T_F)(\gamma)  + (-1)^{\dim F}    \det \rho_F(\gamma) G(F,uv)(\gamma)  ]    \\
 &+ (1/uv) \sum_{ \substack{ 0 \ne \tilde{F}  \subseteq C \\ \gamma \cdot \tilde{F}  = \tilde{F}  } }
 u^{\dim \tilde{F} } \det \rho_{\tilde{F} }(\gamma)  \tilde{S}_{\Gamma_{\tilde{F}}}(\tilde{F} ,u^{-1}v)(\gamma) \sum_{ \substack{ \tilde{F}  \subseteq F  \\ \gamma \cdot F = F }  } \phi_{F^*}[uv](\gamma) (-1)^{\dim F}    \det \rho_F(\gamma)   G(F/\tilde{F} ,uv)(\gamma).
 \end{align*}
 On the other hand, by Definition~\ref{d:stilde}, for a non-zero face $\tilde{F}$ of $C$, 
\[
\tilde{S}_{\Gamma_{\tilde{F}}}(\tilde{F}^*,t)(\gamma) = \sum_{ \substack{ \tilde{F}  \subseteq F  \\ \gamma \cdot F = F }  }
(-1)^{\dim F - \dim \tilde{F}} \det \rho_{F^*}(\gamma) \phi_{F^*}[t](\gamma) G(F/\tilde{F} ,t)(\gamma).
\] 
Substituting into our previous expression, and using the fact that $\det (\rho_F) \det (\rho_{F^*})  = \det (\rho)$ (Lemma~\ref{l:ratio}), gives 
\begin{align*}
&(1/uv) \sum_{ \substack{ 0 \ne F \subseteq C \\ \gamma \cdot F = F } } \phi_{F^*}[uv](\gamma)
[  E_{\Gamma_F}(T_F)(\gamma)  + (-1)^{\dim F}    \det \rho_F(\gamma) G(F,uv)(\gamma)  ]    \\
 &+ (1/uv)   \det \rho(\gamma) \sum_{ \substack{ 0 \ne \tilde{F}  \subseteq C \\ \gamma \cdot \tilde{F}  = \tilde{F}  } }
   (-u)^{\dim \tilde{F} } \det \rho_{\tilde{F} }(\gamma)  \tilde{S}_{\Gamma_{\tilde{F}}}(\tilde{F} ,u^{-1}v)(\gamma)  \tilde{S}_{\Gamma_{\tilde{F}}}(\tilde{F}^*,uv)(\gamma).
 \end{align*}
Comparing with the right hand side of
\eqref{e:reflexive}, we see that it remains to show that
\[
\det \rho(\gamma) \tilde{S}_\Gamma(\hat{C}, uv)(\gamma) =  \sum_{ \substack{ 0 \ne F \subseteq C \\ \gamma \cdot F = F } } \phi_{F^*}[uv](\gamma)
[  E_{\Gamma_F}(T_F)(\gamma)  + (-1)^{\dim F}    \det \rho_F(\gamma) G(F,uv)(\gamma)  ].
\]
On the other hand, by Definition~\ref{d:stilde},
\[
 \det \rho(\gamma) \tilde{S}_\Gamma(\hat{C}, uv)(\gamma) = \sum_{\gamma \cdot F = F} (-1)^{\dim F}  \det \rho_{F}(\gamma)
\phi_{F^*}[uv](\gamma) G(F,uv)(\gamma).
\]
After comparing these two expressions, we are left with proving that
\[
\phi_{\hat{C}}[uv](\gamma) =  \sum_{ \substack{ 0 \ne F \subseteq C \\ \gamma \cdot F = F } } \phi_{F^*}[uv](\gamma)
 E_{\Gamma_F}(T_F)(\gamma).
\]
By Example~\ref{e:tori} and then Lemma~\ref{l:ratio} and Corollary~\ref{c:technical}, 
\begin{align*}
\sum_{ \substack{ 0 \ne F \subseteq C \\ \gamma \cdot F = F } } \phi_{F^*}[uv](\gamma)
 E_{\Gamma_F}(T_F)(\gamma) &=   \frac{1}{uv - 1} \sum_{ \substack{ 0 \ne F \subseteq C \\ \gamma \cdot F = F } } \phi_{F^*}[uv](\gamma)
 \det(uvI - \rho_F(\gamma) )\\
 &=  \frac{1}{uv - 1} ( (uv)^{d + 1} \phi_{\hat{C}} [ (uv)^{-1} ](\gamma) - \phi_{\hat{C}} [ uv ](\gamma) ).
\end{align*}
By \eqref{e:reflex}, the latter sum is equal to $\phi_{\hat{C}}[uv](\gamma)$, as desired.
\end{proof}



\section{Applications to mirror symmetry}\label{s:mirror}

The goal of this section is to present a representation-theoretic version of Batyrev-Borisov mirror symmetry. 

We continue with the notation of the previous section, and let $C$ and $\hat{C}$ be dual, $\Gamma$-invariant cones over reflexive polytopes $P$ and $P^*$ respectively. We may and will assume that the representation $\rho: \Gamma \rightarrow \GL_{d + 1}(\Z)$ is effective. 
If $Y$ (respectively $Y^*$) denotes the projective toric variety corresponding to the normal fan of $P$ (respectively $P^*$), then $Y$ and $Y^*$ are Fano i.e. their anti-canonical divisors are ample. 
 Let  $X$ (respectively $X^*$) be a $\Gamma$-invariant hypersurface of $Y$ (respectively $Y^*$) which is non-degenerate with respect to $P$ (respectively $P^*$). Then $X$ and $X^*$ are \define{Calabi-Yau 
varieties}. That is, by the adjunction formula, $X$ and $X^*$ have trivial canonical divisors, and, since $Y$ and $Y^*$ have at worst  canonical singularities, $X$ and $X^*$ have at worst canonical singularities \cite[Theorem~1.4]{TevCompactifications}. Batyrev and Borisov proved the following version of mirror symmetry in \cite{BBMirror},
\[
E_{\st}(X; u,v) = (-u)^{d - 1} E_{\st}(X^*; u^{-1},v),
\]
where the stringy invariants of $X$ and $X^*$ are defined by Definition~\ref{d:esi} when $\Gamma$ is trivial. 
In particular, if there exist  crepant, toric resolutions $\tilde{X} \rightarrow X$ and $\tilde{X}^* \rightarrow X^*$,
then $E_{\st}(X) = E(\tilde{X})$, $E_{\st}(X^*) = E(\tilde{X}^*)$, and, by Property~\eqref{e:prop1} of the Hodge-Deligne polynomial, 
\[
\dim H^{p,q}(\tilde{X}) = \dim H^{d - 1 - p,q}(\tilde{X}^*) \: \textrm{  for  } \: 0 \le p,q \le d - 1. 
\]
The following representation-theoretic version of mirror symmetry was conjectured in \cite[Conjecture~9.1]{YoRepresentations}.
\begin{theorem}\label{t:mirror}
Let $\Gamma$ be a finite group, 
with integer-valued representation
$\rho: \Gamma \rightarrow \GL_{d + 1}(\Z)$.
Let $C$ and $\hat{C}$ be dual, $\Gamma$-invariant cones over reflexive polytopes $P$ and $P^*$ respectively, and let $X$ and $X^*$ be corresponding $\Gamma$-invariant, non-degenerate, Calabi-Yau hypersurfaces. Then
\[
E_{\st, \Gamma}(X; u,v) = (-u)^{d - 1} \det(\rho) \cdot  E_{\st, \Gamma}(X^*; u^{-1},v).
\]
\end{theorem}
\begin{proof}
By Corollary~\ref{c:reflexive}, and using the fact that $\det (\rho_F)\det(\rho_{F^*}) = \det(\rho)$ (Lemma~\ref{l:ratio}) and $\det(\rho)^2 = 1$, 
\begin{align*}
E_{\st, \Gamma}(X; u,v) &= \frac{\det(\rho)}{uv} \sum_{ [F] \in C/\Gamma } (-u)^{\dim F}  \Ind^{\Gamma}_{\Gamma_F} \det (\rho_F) \tilde{S}_{\Gamma_F}(F, u^{-1}v) \tilde{S}_{\Gamma_F}(F^*, uv) \\
&= (-u)^{d - 1}  \frac{u}{v} \sum_{ [F] \in C/\Gamma } (-u)^{-\dim F^*}  \Ind^{\Gamma}_{\Gamma_F} \det(\rho_{F^*})  \tilde{S}_{\Gamma_F}(F^*, uv)\tilde{S}_{\Gamma_F}(F, u^{-1}v) \\
&=  (-u)^{d - 1} \det(\rho) \cdot  E_{\st, \Gamma}(X^*; u^{-1},v).
\end{align*}
\end{proof}

Suppose that there exist $\Gamma$-equivariant, crepant, toric resolutions 
$\tilde{X} \rightarrow X$ and $\tilde{X}^* \rightarrow X^*$. 
Then Theorem~\ref{t:mirror}, Proposition~\ref{p:stringy} and  Property~\eqref{e:prop1} of the equivariant Hodge-Deligne polynomial imply that 
\begin{equation}\label{e:food}
H^{p,q}(\tilde{X}) = \det(\rho) \cdot H^{d - 1 - p,q}(\tilde{X}^*)  \in R(\Gamma) \: \textrm{  for  } \: 0 \le p,q \le d - 1. 
\end{equation}
The projection $\tilde{X} \rightarrow \tilde{X}/\Gamma$ induces an isomorphism
 $H^*(\tilde{X})^\Gamma \cong  H^*(\tilde{X}/\Gamma)$ which preserves Hodge structures, where 
$H^*(\tilde{X})^\Gamma$ denotes the $\Gamma$-invariant isotypic component of  $H^*(\tilde{X})$. 
In particular, if $\det(\rho)$ is the trivial representation, then 
 the (possibly singular) varieties $\tilde{X}/\Gamma$ and $\tilde{X}^*/\Gamma$ 
have mirror Hodge diamonds i.e. 
\begin{equation}\label{e:orbimirror}
\dim H^{p,q}(\tilde{X}/\Gamma) = \dim H^{d - 1 - p,q}(\tilde{X}^*/\Gamma). 
\end{equation}
This produces infinitely many new examples of (possibly singular) pairs of varieties with mirror Hodge diamonds. We refer the reader to Section~\ref{s:central} and Section~\ref{s:quintic} for explicit examples. 

\begin{remark}\label{r:sl}
 If 
 $\det(\rho)$ is the trivial representation, then $\Gamma$ acts trivially on 
 $H^{d - 1, 0}(\tilde{X}) = H^0(\tilde{X}, \Omega^{d -1}_{\tilde{X}}) \cong \C$ \cite[Corollary~6.8]{YoRepresentations}. In particular, $\Gamma$ leaves global differential forms of $\tilde{X}$ invariant, and $\tilde{X}/\Gamma$ has trivial canonical divisor. By the Luna Slice Theorem \cite[I.2.1]{AudTorus}, if $x \in \tilde{X}$ is fixed by $\gamma \in \Gamma$, then $\gamma$ acts on the tangent space $T_x \tilde{X}$ with determinant $1$. 
\end{remark}

\begin{question}
Assume that $\det(\rho)$ is the trivial representation. What are the singularities of $\tilde{X}/\Gamma$? When does it have canonical singularities i.e. when is it a Calabi-Yau variety? When does it admit a crepant resolution?
\end{question}

This result suggests the following McKay-type correspondence. Assume that $\det(\rho)$ is the trivial representation. Suppose that $Z \rightarrow \tilde{X}/\Gamma$ and $Z^* \rightarrow \tilde{X}^*/\Gamma$ are crepant resolutions. Then we expect that the Hodge diamonds of the Calabi-Yau manifolds $Z$ and $Z^*$ are mirror. As before, this is a question about the stringy invariants of  $\tilde{X}/\Gamma$ and  $\tilde{X}^*/\Gamma$. Since the latter varieties have orbifold singularities, results of Yasuda \cite{YasMotivic} imply that their stringy invariants equal their orbifold Hodge-Deligne polynomials. That is, the dimension of $H^{p,q}(Z)$ equals the dimension of the graded piece
$H^{p,q}_{\orb}(\tilde{X}^*/\Gamma)$
of the Chen-Ruan orbifold cohomology ring of $\tilde{X}^*/\Gamma$ \cite{CRNew}.  We will formulate a precise conjecture 
(Conjecture~\ref{c:Fantechi}) below. 

\begin{question}
Does there exist a homological mirror symmetry correspondence between $Z$ and $Z^*$?
This is suggested by the idea that the derived category of coherent $\Gamma$-sheaves of $\tilde{X}$ should be  equivalent to the derived category of coherent sheaves on $Z$  (cf. \cite{BKRMcKay}).
\end{question}

In \cite{FGOrbifold}, Fantechi and G\"ottsche introduced a graded non-commutative ring $H^*(Y, G)$ with $G$-action associated to an (effective) action of a finite group $G$ on a complex, compact manifold $Y$.
We describe $H^*(Y, G)$ as a graded, complex $G$-representation below, and refer the reader to  Section~1 in \cite{FGOrbifold} for details. 
As an ungraded complex vector space,
\[
H(Y,G) := \bigoplus_{g \in G} H^*(Y^g),
\]
where $Y^g$ denotes the fixed locus of $Y$, a smooth, possibly disconnected, submanifold (see, for example, \cite[Lemma~4.1]{DonLefschetz}). An element $h \in G$ induces an isomorphism $h: Y^g \rightarrow Y^{hgh^{-1}}$, and hence an isomorphism $h_*: H^*(Y^g) \rightarrow H^*(Y^{hgh^{-1}})$.  The latter isomorphisms combine to give an action of $G$ on $H(Y,G)$. Suppose that  $g \in G$ fixes $x \in Y$ and acts on the tangent space 
$T_x Y$ with eigenvalues $\{ e^{2 \pi i \alpha_j} \}_j$ for some $0 \le \alpha_j < 1$. Then the \emph{age} $a(g,x) = a(g,T)$ is equal to 
$\sum_j \alpha_j \in \Q$, and only depends on the connected component $T$ of $x$ in $Y^g$. The Hodge structure on $H^*(Y,G)$ is given by the identification
\[
H(Y,G) = \bigoplus_{g \in G} \bigoplus_{ T \subseteq Y^g} \bigoplus_{p,q} H^{p,q} (Y^g)(-a(g,T)),
\]
where $T$ runs over the connected components of $Y^g$, and $H^{p,q} (Y^g)(-a(g,T))$ has type 
$(p + a(g,T), q + a(g,T))$.  The latter identification induces a $\Q$-grading on $H(Y,G)$.  Observe that the action of $G$ on $H(Y,G)$ preserves the Hodge structure, and hence we may regard the $(p,q)^{\textrm{th}}$ component $H^{p,q}(Y,G)$ of $H(Y,G)$ as a $G$-representation. 

\begin{remark}\label{r:det1}
Suppose that $g \in G$ fixes $x \in Y$ and acts on the tangent space 
$T_x Y$ with determinant $1$. Then the age $a(g,x)$ is a non-negative integer. By Remark~\ref{r:sl}, if $\det(\rho)$ is the trivial representation, then
\[
H(\tilde{X}, \Gamma) = \bigoplus_{p,q} H^{p,q}(\tilde{X}, \Gamma), 
\]
where $p,q$ are non-negative integers. 
\end{remark}

\begin{remark}\cite[Remark~1.4]{FGOrbifold}\label{r:orbifold}
We have an isomorphism of Hodge structures between the invariant subspace $H(Y,G)^G$ and the Chen-Ruan orbifold cohomology ring \cite{CRNew}
\[
H^*_{\orb}(Y/G) = \bigoplus_{g \in S} H^*(Y^g/C(g)),
\]
where $S$ is a set of conjugacy class representatives of $G$, and $C(g)$ denotes the centralizer of $g$ in $G$. In particular, if $Z \rightarrow Y/G$ is a crepant resolution, then we have an isomorphism of Hodge structures $H^*(Z) \cong H(Y/G)^G$.
\end{remark}

\begin{conjecture}\label{c:Fantechi}
Assume that $\det(\rho)$ is the trivial representation, and  there exist  crepant, equivariant,  toric resolutions $\tilde{X} \rightarrow X$ and $\tilde{X}^* \rightarrow X^*$ of the mirror Calabi-Yau varieties $X$ and $X^*$ associated to the $d$-dimensional reflexive polytopes $P$ and $P^*$ respectively. Then we have an isomorphism of $\Gamma$-representations
\[
H^{p,q}(\tilde{X},\Gamma) \cong H^{d - 1 - p,q}(\tilde{X}^*,\Gamma),
\] 
for all pairs of non-negative integers $p,q$. In particular, $H^{p,q}_{\orb}(\tilde{X}/\Gamma) \cong H^{d - 1 - p,q}_{\orb}(\tilde{X}^*/\Gamma)$, and if there exist crepant resolutions $Z \rightarrow \tilde{X}/\Gamma$ and $Z^* \rightarrow \tilde{X}^*/\Gamma$, then $Z$ and $Z^*$ are $(d - 1)$-dimensional Calabi-Yau manifolds with mirror Hodge diamonds.
\end{conjecture}

In the remainder of this section, we present a weaker form of the above conjecture. We will need the following well-known lemma (see, for example, \cite[(30)]{CMSSEquivariant}), which is extremely useful in computations. 

\begin{lemma}[Lefschetz fixed-point formula]\label{l:fixed}
Suppose that a finite group $G$ acts on a manifold $Y$. Consider the Euler characteristic $\chi(Y) =
\sum_i (-1)^i H^i(Y)$ as an element of the complex representation ring $R(G)$. Then, for each $g$ in $G$,  the topological Euler characteristic $\chi(Y^g)$ of the fixed locus $Y^g$ is equal to the evaluation $\chi(Y)(g)$ of the associated virtual character of $\chi(Y)$ at $g$. 
\end{lemma}

Suppose that $Y$ is a compact manifold with an (effective) action of a finite group $G$. Inspired by physics, Dixon et. al. introduced the \emph{orbifold Euler number} in \cite{DHVWStrings},
\begin{equation*}
\chi(Y,G) := \frac{1}{|G|} \sum_{gh = hg} \chi(Y^g \cap Y^h),
\end{equation*}
where the sum runs over all commuting pairs $(g,h)$, and  $\chi(Y^g \cap Y^h)$ is the topological Euler characteristic of $Y^g \cap Y^h$. 
Using the Lefschetz fixed point and after a short computation, this may be rewritten  as 
\begin{equation}\label{e:orby}
\chi(Y,G) = \sum_{g \in S} \chi(Y^g/C(g)),
\end{equation}
where $S$ is a set of conjugacy class representatives of $G$, and $C(g)$ denotes the centralizer of $g$ in $G$. In particular, suppose that whenever $g \in G$ fixes $x \in Y$, then $g$ acts on the tangent space 
$T_x Y$ with determinant $1$. By Remark~\ref{r:det1} and Remark~\ref{r:orbifold},
\[
\chi(Y,G) = \sum_i (-1)^i \dim H_{\orb}^i(Y/G), 
\]
and if $Z \rightarrow Y/G$ is a crepant resolution, then $\chi(Z) = \chi(Y,G)$. Conjecture~\ref{c:Fantechi} would immediately imply the following conjecture.

\begin{conjecture}\label{c:Euler}
Assume that $\det(\rho)$ is the trivial representation, and  there exist  crepant, toric resolutions $\tilde{X} \rightarrow X$ and $\tilde{X}^* \rightarrow X^*$ of the mirror Calabi-Yau varieties $X$ and $X^*$ associated to the $d$-dimensional  reflexive polytopes $P$ and $P^*$ respectively. Then
\[
\chi(\tilde{X},\Gamma) = (-1)^{d - 1}   \chi(\tilde{X}^*,\Gamma).
\]
In particular, if there exist crepant resolutions $Z \rightarrow \tilde{X}/\Gamma$ and $Z^* \rightarrow \tilde{X}^*/\Gamma$, then $\chi(Z) = (-1)^{d - 1}\chi(Z^*)$.  
\end{conjecture}

We verify this conjecture under a strong additional assumption, which holds, for example, when $\Gamma$ is a cyclic group of prime order. 



\begin{corollary}
Assume that $\det(\rho)$ is the trivial representation, and  there exist  crepant, toric resolutions $\tilde{X} \rightarrow X$ and $\tilde{X}^* \rightarrow X^*$ of the mirror Calabi-Yau varieties $X$ and $X^*$ associated to the $d$-dimensional  reflexive polytopes $P$ and $P^*$ respectively. 
Further assume that for each non-trivial element $\gamma \in \Gamma$, the centralizer of $\gamma$ satisfies $C(\gamma) = \langle \gamma \rangle$. Then
\[
\chi(\tilde{X},\Gamma) = (-1)^{d - 1}   \chi(\tilde{X}^*,\Gamma).
\]
\end{corollary}
\begin{proof}
By \eqref{e:orbimirror}, $\chi(\tilde{X}/\Gamma)  = (-1)^{d - 1} \chi(\tilde{X}^*/\Gamma)$. 
Hence, by \eqref{e:orby} and our assumption, it will be enough to show that for all non-trivial $\gamma$ in $\Gamma$, 
\[
\chi(\tilde{X}^\gamma) = (-1)^{d - 1} \chi((\tilde{X}^*)^\gamma).
\]
By the  Lefschetz fixed-point formula (Lemma~\ref{l:fixed}), we are reduced to verifying that 
\[
\chi(\tilde{X})(\gamma) = (-1)^{d - 1} \chi(\tilde{X}^* )(\gamma).
\]
The latter equality is a direct consequence of \eqref{e:food}. 
\end{proof}

\section{Centrally symmetric reflexive polytopes}\label{s:central}

In this section, we apply our results to centrally symmetric reflexive polytopes, and establish Conjecture~\ref{c:Fantechi} in this case. 

We continue with the notation of the previous section. That is, the representation $\rho: \Gamma \rightarrow \GL_{d + 1} (\Z)$ is isomorphic to the direct sum of the trivial representation and $\rho': \Gamma \rightarrow \GL(M)$, and 
$C$ is the cone over a reflexive, $\Gamma$-invariant lattice polytope $P$.  Throughout this section, we set $\Gamma = \Z_2 = \langle \epsilon \rangle$ and $\rho'(\epsilon) = -I$. In this case, $P$ is called  
\emph{centrally symmetric}, and $\det(\rho)$ is the trivial representation if and only if $d$ is even.  There exist $\Z_2$-invariant hypersurfaces $X$ and $X^*$ that are non-degenerate with respect to $P$ and $P^*$ respectively \cite[Corollary~7.8, Corollary~7.10, Section~11]{YoEquivariant}.

Our first goal is to calculate $E_{\st, \Z_2}(X)(\epsilon)$ using Corollary~\ref{c:reflexive}. Firstly,  one calculates  that (see Section~11 in \cite{YoEquivariant}), 
\[
\phi_C[t](\epsilon) = (1 + t)^{d}.
\]
Also, by Example~\ref{e:cc},
\[
H(C,t)(\epsilon) = (1 + t)^d.
\]
By Definition~\ref{d:stilde} and the above computations,
\[
\tilde{S}_{\Z_2}(C,t)(\epsilon) = (-1)^{d + 1} G(\hat{C},t)(\epsilon) + (-1)^d \phi_C[t](\epsilon) = (-1)^d t\alpha_d(t),
\] 
where $\alpha_d(t) = \sum_{i = 0}^{d - 1} a_i t^i$ is the unique polynomial 
 of degree $d - 1$ with 
$\alpha_d(t) = t^{d - 1}\alpha_d(t^{-1})$ and $a_i = \binom{d}{i}$ for $i \le \lfloor \frac{d - 1}{2} \rfloor$. 
Hence, by Corollary~\ref{c:reflexive},
\begin{align*}
E_{\st,  \Z_2}(X; u,v)(\epsilon) &= \frac{(-1)^d}{uv} ( \tilde{S}_{\Z_2}(\hat{C}, uv)(\epsilon) - u^{d + 1}   \tilde{S}_{\Z_2}(C, u^{-1}v)(\epsilon) ) \\
&=  \alpha_d(uv) - u^{d - 1} \alpha_d(u^{-1} v).
\end{align*}
Let $\zeta$ denote the non-trivial character of $\Z_2$. Using the latter computation, we may compute the equivariant stringy invariant of $X$ in terms of the usual stringy invariant of $X$ since
\[
E_{\st,  \Z_2}(X) = \frac{E_{\st} (X) + E_{\st,  \Z_2} (X)(\epsilon)  }{2} + 
\frac{E_{\st} (X) - E_{\st,  \Z_2} (X)(\epsilon)  }{2}\zeta.
\]

For the remainder of the section we assume there exist  $\Z_2$-equivariant, crepant, toric resolutions $\tilde{X} \rightarrow X$ and $\tilde{X}^* \rightarrow X^*$. By Proposition~\ref{p:stringy}, 
\[
E_{\Z_2}(\tilde{X}; u,v) = \sum_{p,q} (-1)^{p + q} H^{p,q}(\tilde{X}) u^p v^q =  E_{\st,  \Z_2}(X; u,v) \in R(\Z_2)[u,v]. 
\]
We claim that $\Z_2$ acts freely on $\tilde{X}$ and $\tilde{X}^*$. Indeed, the only fixed cone of a $\Z_2$-invariant fan is the origin, and hence the only fixed points of the ambient toric variety are the $2^d$ fixed points on the torus. Explicitly, if $T \cong (\C^*)^d$, then $\Z_2$ acts sending $(x_1, \ldots, x_d)$ to 
$(x_1^{-1}, \ldots, x_d^{-1})$, and the fixed points are $(\pm 1, \ldots, \pm 1)$. By the Lefschetz fixed point theorem (Lemma~\ref{l:fixed}), the number of fixed points of $\tilde{X}$ equals 
$\chi(\tilde{X}^\epsilon) = \chi(\tilde{X})(\epsilon)$. On the other hand, by the above calculation,
\[
\chi(\tilde{X})(\epsilon) = E_{\Z_2}(\tilde{X}; 1,1)(\epsilon) = E_{\st,  \Z_2}(X; 1,1)(\epsilon)   = \alpha_d(1) - \alpha_d(1) = 0. 
\]
Hence, with the notation of Section~\ref{s:mirror}, we have an equality of graded $\Z_2$-representations 
$H(\tilde{X}, \Z_2) =  H^*(\tilde{X})$. If $d$ is even, then $\det(\rho)$ is trivial, and Theorem~\ref{t:mirror} implies that 
\[
E_{\Z_2}(\tilde{X}; u,v)= (-u)^{d - 1} E_{\Z_2}(\tilde{X}^*; u^{-1},v),
\]
and hence we have isomorphisms of $\Z_2$-representations
\[
H^{p,q}(\tilde{X}) = H^{p,q}(\tilde{X},\Z_2) \cong H^{d - 1 - p,q}(\tilde{X}^*,\Gamma) = H^{p,q}(\tilde{X}^*).
\] 
This verifies that Conjecture~\ref{c:Fantechi} holds when $d$ is even.

Explicitly, it follows that the Hodge numbers of the manifold $\tilde{X}/\Z_2$ are determined by the symmetries $h^{p,q}(\tilde{X}/\Z_2) = h^{q,p}(\tilde{X}/\Z_2) = h^{d - 1 - p,d - 1 - q}(\tilde{X}/\Z_2)$, together with
\[
h^{p,q}(\tilde{X}/ \Z_2) = \left\{ \begin{array}{ll}
\frac{h^{p,q}(\tilde{X})}{2} & \textrm{  if  } p \ne q, p + q \ne d - 1 \textrm{ or }  p = q,  p + q = d - 1 \\
\frac{h^{p,q}(\tilde{X}) + \binom{d}{p} }{2} & \textrm{  if  }  p =  q < \frac{d - 1}{2} \\
\frac{h^{p,q}(\tilde{X}) + (-1)^d \binom{d}{p} }{2}  & \textrm{  if  }  p + q = d - 1, p < q.
\end{array} \right.
\]
Similarly, we compute the Hodge numbers of the manifold $\tilde{X}^*/\Z_2$. Since $\tilde{X}$ and $\tilde{X}^*$ have mirror Hodge diamonds by Batyrev-Borisov duality, we see explicitly that $\tilde{X}/\Z_2$ and 
$\tilde{X}^*/\Z_2$ have mirror Hodge diamonds when $d$ is even. 

\begin{remark}
Since $\Z_2$ acts freely on $\tilde{X}$, the topological Euler characteristic of $\tilde{X}/\Z_2$ is $\chi(\tilde{X}/ \Z_2) = \frac{\chi(\tilde{X})}{2}$.
\end{remark}

\begin{remark}
When $d$ is odd, the canonical divisor $K$ of $\tilde{X}/ \Z_2$ is non-trivial, but $2K$ is trivial.   
\end{remark}

\begin{example}
When $d = 3$, $\tilde{X}$ is a K$3$-surface with Hodge diamond determined by  $h^{2,0}(\tilde{X}) = 1$, $h^{1,0}(\tilde{X}) = 0$ and $h^{1,1}(\tilde{X}) = 20$. The quotient  $\tilde{X}/\Z_2$ is an Enriques surface with Hodge diamond determined by $h^{2,0}(\tilde{X}/\Z_2) = h^{1,0}(\tilde{X}/\Z_2)  =  0$ and $h^{1,1}(\tilde{X}/\Z_2) = 10$. 
\end{example}

\begin{remark}
If $p =  q < \frac{d - 1}{2}$, then 
\[
H^{p,q}(\tilde{X}) = \frac{h^{p,q}(\tilde{X}) + \binom{d}{p} }{2} +  \frac{h^{p,q}(\tilde{X}) - \binom{d}{p} }{2}\zeta \in R(\Z_2). 
\]
If  $p + q = d - 1, p < q$, then 
\[
H^{p,q}(\tilde{X}) = \frac{h^{p,q}(\tilde{X}) + (-1)^d\binom{d}{p} }{2} +  \frac{h^{p,q}(\tilde{X}) - (-1)^d\binom{d}{p} }{2}\zeta \in R(\Z_2). 
\]
We conclude that if $p = q <  \frac{d - 1}{2}$ or $p + q = d - 1, p < q$, but not both, then we have a lower bound 
$h^{p,q}(\tilde{X}) \ge \binom{d}{p}$. Moreover, if $p = q <  \frac{d - 1}{2}$ or $p + q = d - 1, p < q$ and $d$ is even, then  $h^{p,q}(\tilde{X} / \Z_2) \ge \binom{d}{p}$.  
We will see in Example~\ref{e:crosspoly} below that each of these lower bounds can be obtained. 
\end{remark}

\begin{example}\label{e:crosspoly}
Consider the centrally symmetric, reflexive polytope $P = [-1,1]^d$. Then $X$ is a smooth hypersurface of the product $Y$ of $\P^1$ with itself $d$ times, embedded in $\P^{3^d - 1}$ (and hence $X = \tilde{X}$). By the Lefschetz hyperplane theorem, the restriction map $H^i Y \rightarrow H^i X$ is an isomorphism for $i < d - 1$. We deduce that for $p + q < d - 1$,
\[
h^{p,q}(X) = h^{p,q}(X/ \Z_2) = \left\{ \begin{array}{ll}
 &\binom{d}{p}  \; \textrm{  if  } p = q  \\
& \; 0  \; \; \; \textrm{  otherwise.  }  
\end{array} \right.
\]
The polytope $P^*$ is the $d$-dimensional cross-polytope i.e. if $e_1, \ldots, e_d$ is a basis of the lattice, then $P^*$ is the convex hull of $\{ \pm e_i \}_{1 \le i \le d}$. The hypersurface $X^*$ is singular in general, but we can construct a $\Z_2$-equivariant resolution $\tilde{X}^* \rightarrow X^*$. Indeed, $X^*$ is a non-degenerate hypersurface in the toric variety with fan given by the fan over the faces of $P$. One may construct a $\Z_2$-equivariant, regular, unimodular, lattice  triangulation of the boundary of $P$, which induces the corresponding toric, crepant resolution. By Batyrev-Borisov duality, for $p < q$, 
\[
h^{p,q}(\tilde{X}^*) = \left\{ \begin{array}{ll}
 &\binom{d}{p}  \; \textrm{  if  } p + q = d - 1  \\
& \; 0  \; \; \; \textrm{  otherwise.  }  
\end{array} \right.
\]
It follows that for $p < q$,
\[
h^{p,q}(\tilde{X}^*/ \Z_2) = \left\{ \begin{array}{ll}
 &\binom{d}{p}  \; \textrm{  if  } p + q = d - 1 \textrm{ and } $d$ \textrm{ is even }  \\
& \; 0  \; \; \; \textrm{  otherwise.  }  
\end{array} \right.
\]

\end{example}

\begin{example}
If  $P = [-1,1]^4$,  then $X$ is a smooth hypersurface of $\P^1 \times \P^1 \times \P^1 \times \P^1 \subseteq \P^{80}$. We calculate that the $h^*$-polynomial of $P$ is $h^*_P(t) = 1 + 76t + 230t^2 + 76t^3 + t^4$ and the $h^*$-polynomial of $P^*$ is $h^*_{P^*}(t) = (1 + t)^4$ (cf. \cite{BraEhrhart}).  
Moreover, $P$ has $16$ vertices, $32$ edges with $h^*(t) = 1 + t$, $24$ $2$-dimensional faces with $h^*(t) = 1 + 6t + t^2$,  and $8$ facets with $h^*(t) = 1 + 23t + 23t^2 + t^3$. The $h$-polynomial and $g$-polynomial of $C$ 
are equal to $h_C(t) = 1 + 12t + 14t^2 + 12t^3 + t^4$ and 
$g_C(t) = 1 + 11t + 2t^2$ respectively. 
For every non-zero face $F$ of $C$, $g_{F^*}(t) = 1$ and $h^*_{F^*}(t) = 1$, and if $F$ is a proper face, then $\tilde{S}(F^*, t) = 0$. Also, $h_{\hat{C}}(t) = (1 + t)^4$ and $g_{\hat{C}}(t) = 1 + 3t + 2t^2$.
We deduce that $\tilde{S}(C,t) = t + 68t^2 + 68t^3 + t^4$, $\tilde{S}(\hat{C},t) = t + 4t^2 + 4t^3 + t^4$, and the respective Hodge diamonds of $X$ and $\tilde{X}^*$ are equal to  
\begin{center}
\begin{tabular}{ c    c     c   c  c  c  c                 c c  c                               c    c     c   c  c  c  c         }
 &  &  & $1$ & &                                             & &&      &        &  &  & $1$ & &   \\
  &  & $0$ & & $0$ &  &                                & &&           &  & $0$ & & $0$ &  &\\
   &  $0$  &  & $4$ & & $0$ &                    &  &&         &  $0$  &  & $68$ & & $0$ &\\
     $1$  &   & $68$  &  & $68$ &  & $1$      & &&        $1$  &   & $4$  &  & $4$ &  & $1$ \\
      &  $0$  &  & $4$ & & $0$ &              &&   &       &  $0$  &  & $68$ & & $0$ &  \\
        &  & $0$ & & $0$ &  &                         &&   &          &  & $0$ & & $0$ &  &\\
        &  &  & $1$ & &                               &          &   &&            &  &  & $1$ & & \\
\end{tabular}
\end{center}
with topological Euler characteristics $-128$ and $128$ respectively.  The respective Hodge diamonds of the quotients $X/\Z_2$ and $\tilde{X}^*/\Z_2$ are equal to  
\begin{center}
\begin{tabular}{ c    c     c   c  c  c  c                 c c  c                               c    c     c   c  c  c  c         }
 &  &  & $1$ & &                                             & &&      &        &  &  & $1$ & &   \\
  &  & $0$ & & $0$ &  &                                & &&           &  & $0$ & & $0$ &  &\\
   &  $0$  &  & $4$ & & $0$ &                    &  &&         &  $0$  &  & $36$ & & $0$ &\\
     $1$  &   & $36$  &  & $36$ &  & $1$      & &&        $1$  &   & $4$  &  & $4$ &  & $1$ \\
      &  $0$  &  & $4$ & & $0$ &              &&   &       &  $0$  &  & $36$ & & $0$ &  \\
        &  & $0$ & & $0$ &  &                         &&   &          &  & $0$ & & $0$ &  &\\
        &  &  & $1$ & &                               &          &   &&            &  &  & $1$ & & \\
\end{tabular}
\end{center}
with topological Euler characteristics $-64$ and $64$ respectively.  
\end{example}




\section{The quintic threefold}\label{s:quintic}

The goal of this section is to apply our results to Fermat hypersurfaces, and explicitly verify Conjecture~\ref{c:Fantechi} for the quintic threefold. 

We continue with the notation of the previous section, and let $\Gamma$ be a subgroup of $\Sym_{d + 1}$ 
acting on $\R^{d + 1}$ via the standard representation. Let $C \subseteq \R^{d + 1}$ be the first quadrant, and let $P$ be the convex hull of $(d + 1, 0, \ldots, 0), \ldots, 
(0, \ldots, 0, d + 1)$ i.e. the $(d + 1)^{\textrm{st}}$ dilate of the standard simplex. Replace $\Z^{d + 1}$ with the lattice generated by all lattice points in the affine span of $P$, and let $M$ be the translation of $\aff(P) \cap \Z^{d + 1}$ to the origin by the unique interior lattice point $(1,\ldots, 1)$ of $P$. Then $P$ is a $\Sym_{d + 1}$-invariant, reflexive polytope, and the Fermat hypersurface $X = \{ x_0^{d + 1} + \cdots + x_d^{d + 1} = 0 \} \subseteq \P^{d}$ is a smooth, $\Sym_{d + 1}$-invariant, Calabi-Yau hypersurface which is non-degenerate with respect to $P$. Moreover, the induced action of $\Sym_{d + 1}$ on $H^* X$ is explicitly computed in \cite[Example~8.4]{YoRepresentations}. 

The dual polytope $P^*$ is the standard simplex in $\R^{d + 1}$ with lattice $\Z^{d + 1} + \Z(\frac{1}{d + 1}, \ldots, \frac{1}{d + 1})$. After choosing co-ordinates, a $\Sym_{d + 1}$-invariant hypersurface of the torus  has the form $\{ x_1 + \cdots + x_d + \frac{1}{x_1\cdots x_d} +  \psi = 0 \} \subseteq  (\C^*)^d$ for some $\psi \in \C^*$. For a general choice of $\psi$, the corresponding hypersurface is non-degenerate with represent to $P^*$.  The normal fan to $P^*$ equals the fan over the faces of $P$ in $M_\R$, and a $\Sym_{d + 1}$-equivariant, regular, unimodular triangulation of the boundary of $P$ (which one can verify exists) induces a $\Sym_{d + 1}$-equivariant, toric, crepant resolution 
$\tilde{X}^* \rightarrow X^*$. 

By \eqref{e:orbimirror}, if $\Gamma \subseteq \Sym_{d + 1}$ 
is the alternating group $A_{d + 1}$ of even permutations, then the orbifolds $X/A_{d + 1}$ and $\tilde{X}^*/A_{d + 1}$ have mirror Hodge diamonds. 
For the remainder of the section, we will specialize to the case when $d = 4$, and show that the
Hodge diamonds of $H(X,A_{5})$ and $H(\tilde{X}^*, A_{5})$ are mirror, hence verifying Conjecture~\ref{c:Fantechi}  in this case. 

Consider the alternating group $A_5$. It has $60$ elements and $5$ conjugacy classes. Explicitly, the conjugacy class of the identity element consists of one element. The conjugacy class containing $\gamma = (12)(34)$ has $15$ elements and centralizer $C(\gamma) \cong \Z_2 \times \Z_2$ with $C(\gamma)/ \langle \gamma \rangle$ generated by $(13)(24)$.  The conjugacy class containing $\gamma = (123)$ has $20$ elements and $C(\gamma) = \langle \gamma \rangle$. 
There are $2$ conjugacy classes consisting of cycles of order $5$, both with $12$ elements and  $C(\gamma) = \langle \gamma \rangle$, where $\gamma$ is a conjugacy class representative. 
Recall that 
\[
H(X, A_5) = \bigoplus_{\gamma \in A_5} H^*(X^\gamma),  \: \: \: H(\tilde{X}^*, A_5) = \bigoplus_{\gamma \in A_5} H^*((\tilde{X}^*)^\gamma),
\]
with the age grading and $A_5$-action described in Section~\ref{s:mirror}. We will compute the $A_5$-representations $H(X, A_5)$ and $H(\tilde{X}^*, A_5)$  below. 

Let $\mu = 1 + 2 \Ind_{\Z_2}^{A_5} 1 + 2  \Ind_{\Z_3}^{A_5} 1 \in R(A_5)$ be the $101$-dimensional permutation representation corresponding to the action of $A_5$ on the set 
\[
\{ (b_1, \ldots, b_5) \mid 0 \le b_i \le 4, \sum_i b_i = 5 \}.  
\]
It was shown in \cite[Figure~1]{YoRepresentations} that the 
representations of $A_5$ on the cohomology of $X$ and $\tilde{X}^*$ are described by the respective diamonds of representations below

\begin{center}
\begin{tabular}{ c    c     c   c  c  c  c                 c c  c                               c    c     c   c  c  c  c         }
 &  &  & $1$ & &                                             & &&      &        &  &  & $1$ & &   \\
  &  & $0$ & & $0$ &  &                                & &&           &  & $0$ & & $0$ &  &\\
   &  $0$  &  & $1$ & & $0$ &                    &  &&         &  $0$  &  & $\mu$ & & $0$ &\\
     $1$  &   & $\mu$  &  & $\mu$ &  & $1$      & &&        $1$  &   & $1$  &  & $1$ &  & $1$. \\
      &  $0$  &  & $1$ & & $0$ &              &&   &       &  $0$  &  & $\mu$ & & $0$ &  \\
        &  & $0$ & & $0$ &  &                         &&   &          &  & $0$ & & $0$ &  &\\
        &  &  & $1$ & &                               &          &   &&            &  &  & $1$ & & \\
\end{tabular}
 \end{center}

Consider the element $\gamma = (12)(24)$ in $A_5$. The fixed locus $X^\gamma$ consists of the degree $5$ curve $C = X \cap \{ (x: x: y : y : z) \} \subseteq \P^2$ together with $\{ (x: - x: y : - y: 0 ) \} \cong \P^1$ (cf. \cite[Lemma~2.3]{CheRepresentations}). In both cases, $C(\gamma)/\langle \gamma \rangle \cong \Z_2$ acts by exchanging $x$ and $y$. 
Consider the action of $\Z_2$ on $\R^2$ by exchanging co-ordinates, and let $Q$ be the convex hull of the origin, $(5,0)$ and $(0,5)$ in $\R^2$. 
Then $C \subseteq \P^2$ may be viewed as a $\Z_2$-invariant curve which is non-degenerate with respect to $Q$. In particular, by  Corollary~6.8 in \cite{YoRepresentations}, the $\Z_2$-representation
$H^{1,0}(C)$ is the permutation representation associated to the action of $\Z_2$ on the $6$ interior lattice points in $Q$. Similarly, $\Z_2$ acts trivially on $H^*(\P^1)$. 
One verifies that for every fixed point $x \in X^\gamma$, $\gamma$ acts on the tangent space  $T_x X$ with eigenvalues $\{ 1, -1, -1 \}$. In particular,  
the age $a(g,T)$ of both connected components $T$ of $X^\gamma$ equals $1$. We conclude that the
 representation of $A_5$ on  $\bigoplus_{\gamma'} H^*(X^{\gamma'})$, where $\gamma'$ varies over the elements of the conjugacy class of $(12)(34)$, is isomorphic to 

\begin{center}
\begin{tabular}{ c    c     c   c  c  c  c                 c c  c                               c    c     c   c  c  c  c         }
 &  &  & $0$ & &                                               \\
  &  & $0$ & & $0$ &  &                              \\
   &  $0$  &  & $\nu'$ & & $0$ &                  \\
     $0$  &   & $\nu$  &  & $\nu$ &  & $0$ \:, \\
      &  $0$  &  & $\nu'$ & & $0$ &               \\
        &  & $0$ & & $0$ &  &                     \\
        &  &  & $0$ & &                             
\end{tabular}
 \end{center}
where $\nu= 2 \Ind_{\Z_2}^{A_5} 1 + 2  \Ind_{\Z_2 \times \Z_2}^{A_5} 1$ and
$\nu' = 2  \Ind_{\Z_2 \times \Z_2}^{A_5} 1$. 

Consider the element $\gamma = (123)$ in $A_5$. The fixed locus $X^\gamma$ consists of the degree $5$ curve $C = X \cap \{ (x: x: x : y : z) \} \subseteq \P^2$, together with 
the two fixed points $x_1 = (1: e^{2 \pi i /3} : e^{4 \pi i /3}  : 0 : 0)$ and $x_2 = (1: e^{4 \pi i /3} : e^{2 \pi i /3}  : 0 : 0)$.
One verifies that $\gamma$ acts on the tangent space  $T_x X$ as $e^{2 \pi i /3}$ times the identity transformation when $x = x_1$, as  $e^{4 \pi i /3}$ times the identity transformation when $x = x_2$, and acts with eigenvalues $\{ 1, e^{2 \pi i/3}, e^{4 \pi i/3}\}$ when $x \in C$.  We conclude that the
 representation of $A_5$ on  $\bigoplus_{\gamma'} H^*(X^{\gamma'})$, where $\gamma'$ varies over the elements of the conjugacy class of $(123)$, is isomorphic to 

\begin{center}
\begin{tabular}{ c    c     c   c  c  c  c                 c c  c                               c    c     c   c  c  c  c         }
 &  &  & $0$ & &                                               \\
  &  & $0$ & & $0$ &  &                              \\
   &  $0$  &  & $\theta'$ & & $0$ &                  \\
     $0$  &   & $\theta$  &  & $\theta$ &  & $0$ \:, \\
      &  $0$  &  & $\theta'$ & & $0$ &               \\
        &  & $0$ & & $0$ &  &                     \\
        &  &  & $0$ & &                             
\end{tabular}
 \end{center}
where $\theta = 6 \Ind_{\Z_3}^{A_5} 1$ and
$\theta' = 2  \Ind_{\Z_3}^{A_5} 1$.  
 
 Consider the element $\gamma = (12345)$ in $A_5$. One verifies that $\gamma$ has precisely $5$ fixed points in $\P^4$, none of which lie on $X$.  
 \begin{remark}\label{r:nonabelian}
Observe that the stabilizer of each of the $5$ elements in $X \cap \{ (x: x: x : x : y) \}$ is the subgroup $A_4 \subseteq A_5$. By the Luna Slice Theorem \cite[I.2.1]{AudTorus}, the images of these points in $X/A_5$ are singularities of the form $\C^3/A_4$.  In particular, $X/A_5$ has non-abelian singularities. 
\end{remark}
 We conclude that the $A_5$-representation $H(X, A_5)$ is described by the diamond of representations 
 \begin{center}
\begin{tabular}{ c    c     c   c  c  c  c                 c c  c                               c    c     c   c  c  c  c         }
 &  &  & $0$ & &                                               \\
  &  & $0$ & & $0$ &  &                              \\
   &  $0$  &  & $\Phi'$ & & $0$ &                  \\
     $0$  &   & $\Phi$  &  & $\Phi$ &  & $0$ \:, \\
      &  $0$  &  & $\Phi'$ & & $0$ &               \\
        &  & $0$ & & $0$ &  &                     \\
        &  &  & $0$ & &                             
\end{tabular}
 \end{center}
 where $\Phi = 1  + 4\Ind_{\Z_2}^{A_5} 1  + 8 \Ind_{\Z_3}^{A_5} 1 + 2\Ind_{\Z_2 \times \Z_2}^{A_5} 1$ and
$\Phi' = 1 +  2  \Ind_{\Z_3}^{A_5} 1 +  2  \Ind_{\Z_2 \times \Z_2}^{A_5} 1$.  In particular, by taking invariant subspaces of the above representations, we calculate the orbifold Hodge diamond of $X/A_5$
to be
 \begin{center}
\begin{tabular}{ c    c     c   c  c  c  c                 c c  c                               c    c     c   c  c  c  c         }
 &  &  & $1$ & &                                               \\
  &  & $0$ & & $0$ &  &                              \\
   &  $0$  &  & $5$ & & $0$ &                  \\
     $1$  &   & $15$  &  & $15$ &  & $1$. \\
      &  $0$  &  & $5$ & & $0$ &               \\
        &  & $0$ & & $0$ &  &                     \\
        &  &  & $1$ & &                             
\end{tabular}
 \end{center}
We next consider the equivariant toric resolution $\tilde{X}^* \rightarrow X^*$. We may view $P^*$ as the image of the standard $4$-dimensional simplex  in $N = \Z^{5}/(1,\ldots,1)$, with induced action of $\Sym_{5}$. The decomposition of the corresponding projective toric variety induces a decomposition 
$X^* = \cup_{F^*} (X^* \cap T_{F^*})$, where $T_{F^*}  = \Spec \C[N_{F^*}]$ is the torus orbit corresponding to the non-empty face $F^*$ of $P^*$, and $N_{F^*}$ is a translation to the origin of the intersection of $N$ with the affine span of $F^*$, with induced action of the stabilizer $\Gamma_{F^*}$ of 
$F^*$ (see Section~2 in \cite{YoRepresentations} for details). For any element $\gamma \in \Sym_{5}$ and $\gamma$-invariant face $F^*$,  consider the finite abelian group  $N_{F^*}(\gamma) = N_{F^*}/\langle\gamma \cdot e - e \mid e \in N_{F^*}\rangle$. The induced morphism $\Spec \C[N_{F^*}(\gamma)] \hookrightarrow T_{F^*}$ is the inclusion of the $\gamma$-fixed locus  $T_{F^*}^\gamma$ of $T_{F^*}$.

Consider the element $\gamma = (12)(24)$ in $A_5$. We compute $N(\gamma) = \Z^5/( e_0 + \cdots + e_4 = 0, e_0 = e_1, e_2 = e_3) \cong \Z^2 = \Z e_0 + \Z e_2$, with induced action of $C(\gamma)/\langle \gamma \rangle \cong \Z_2$ by exchanging $e_0$ and $e_2$. 
The fixed locus $X^* \cap T^\gamma \subseteq T^\gamma \cong (\C^*)^2$ is a hypersurface which is non-degenerate with respect to its Newton polytope $Q = \conv\{ e_0, e_2, -2e_0 - 2e_2 \}$. The closure of $X^* \cap T^\gamma$ in $\tilde{X}^*$  is a connected component $C$ of $(\tilde{X}^*)^\gamma$. For every $x \in C$, $\gamma$ acts on $T_x \tilde{X}^*$ with eigenvalues $\{ 1, -1 , -1 \}$ and age $a(\gamma, C) = 1$.  
 By  Corollary~6.8 in \cite{YoRepresentations}, the $\Z_2$-representation
$H^{1,0}(C)$ is the permutation representation associated to the action of $\Z_2$ on the $2$ ($\Z_2$-fixed) interior lattice points in $Q$. We remark that the  image of $C \subseteq \tilde{X}^*$ in $X^*$ contains two singular points of $X^*$, and one verifies that the pre-image of each singular point in $\tilde{X}^*$ contains a unique $\gamma$-fixed point. 

Consider the face $F^*$ of $P^*$ given by the convex hull of the images of $e_0, e_1, e_2, e_3$ in $N$, which corresponds to a ray in the normal fan of $P^*$. 
In this case, $N_{F^*} \cong \Z^3$ is generated by the images of $e_0 - e_3$, $e_1 - e_3$ and $e_2 - e_3$ in $N$, and $N_{F^*}(\gamma) = \Z e_0 + \Z e_2/ (2(e_0 + e_2)) \cong \Z \times \Z_2$, with induced action of $C(\gamma)/\langle \gamma \rangle \cong \Z_2$ by exchanging $e_0$ and $e_2$. 
The closure of $X^* \cap \Spec \C[N_{F^*}(\gamma)] \subseteq T_{F^*}$ in $X^*$ is isomorphic to $\P^1 = \{ (x:y) \}$, with induced action of $\Z_2$ exchanging $x$ and $y$. The two points corresponding to $(0:1)$ and $(1:0)$ in $\P^1$ are singular points of $X^*$, lying in the torus orbits $T_{F^*}$ for $F^* = \langle e_0, e_1 \rangle$ and $F^* = \langle e_2, e_3 \rangle$. The corresponding cones in the normal fan of $P^*$ are singular $3$-dimensional cones, each isomorphic to the cone over $5Q$, where $Q$ is the standard $2$-dimensional simplex. The diagram below shows an equivariant, unimodular, regular triangulation of $5Q$, in which $\gamma$ acts by exchanging the horizontal and vertical co-ordinates, and 
the maximal $\gamma$-invariant faces are numbered. 
\begin{center}
\setlength{\unitlength}{0.7cm}
\begin{picture}(11,5.5)



\put(0,2){\circle*{0.1}}
\put(0,3){\circle*{0.1}}
\put(0,4){\circle*{0.1}}
\put(0,5){\circle*{0.1}}
\put(1,3){\circle*{0.1}}
\put(1,4){\circle*{0.1}}
\put(2,0){\circle*{0.1}}

\put(3,0){\circle*{0.1}}
\put(3,1){\circle*{0.1}}

\put(4,0){\circle*{0.1}}
\put(4,1){\circle*{0.1}}
\put(5,0){\circle*{0.1}}

\linethickness{0.075mm}
\put(0,1){\line(0,1){4}}
\put(1,0){\line(1,0){4}}
\put(0,5){\line(1,-1){2}}
\put(3,2){\line(1,-1){2}}

\put(0,1){\circle*{0.1}}

\put(1,0){\circle*{0.1}}

\put(1,2){\circle*{0.1}}

\put(2,1){\circle*{0.1}}

\put(2,3){\circle*{0.1}}
\put(3,2){\circle*{0.1}}

\put(0,2){\line(1,0){1}}
\put(0,3){\line(1,0){2}}
\put(0,4){\line(1,0){1}}
\put(2,1){\line(1,0){2}}

\put(2,0){\line(0,1){1}}
\put(3,0){\line(0,1){2}}
\put(4,0){\line(0,1){1}}
\put(1,2){\line(0,1){2}}

\put(0,3){\line(1,-1){1}}
\put(0,4){\line(1,-1){4}}

\put(2,1){\line(1,-1){1}}
\put(0,2){\line(1,-1){2}}


\color{blue}{
\put(0,0){\line(1,0){1}}
\put(0,1){\line(1,0){2}}
\put(1,2){\line(1,0){2}}

\put(0,0){\line(0,1){1}}
\put(1,0){\line(0,1){2}}
\put(2,1){\line(0,1){2}}

\put(0,1){\line(1,-1){1}}
\put(1,2){\line(1,-1){1}}
\put(2,3){\line(1,-1){1}}

\put(0,0){\circle*{0.1}}

\put(1,1){\circle*{0.1}}

\put(2,2){\circle*{0.1}}

\put(0.15,0.15){$1$}
\put(0.6,0.5){$2$}
\put(1.15,1.15){$3$}
\put(1.6,1.5){$4$}
\put(2.15,2.15){$5$}

}








\end{picture}

\end{center}
In the corresponding resolution of the ambient toric varieties, the orbit $\C^*$ corresponding to $5Q$ is replaced by torus orbits corresponding to the faces of the triangulation which are not contained in the boundary of $5Q$. The action of $\Z_2$ on $\C^* = \Spec \C[x, x^{-1}]$ sends $x$ to $x^{-1}$ and has fixed points $\{ 1, -1 \}$, with $-1$ the fixed point in $X^*$. One calculates that the $\gamma$-fixed locus of the fiber of $-1 \in X^*$ consists of the disjoint union of a point (corresponding to the maximal face labelled $1$ above) and two $\P^1$'s   (such that $\{ 0, \infty \}$ correspond  to the maximal faces labelled $2$ and $3$, and $4$ and $5$ respectively), each with corresponding age $1$. The action of $C(\gamma)/\langle \gamma \rangle \cong \Z_2$ exchanges the fibers over the two singular points. 

We conclude that $(\tilde{X}^*)^\gamma$ consists of the disjoint union of a genus $2$ curve and $5$ $\P^{1}$'s. Moreover, $C(\gamma)/\langle \gamma \rangle \cong \Z_2$ acts by leaving the genus $2$ curve and one of the $\P^1$'s invariant, and exchanging the other four $\P^1$'s in two orbits. 
The
 representation of $A_5$ on  $\bigoplus_{\gamma'} H^*((\tilde{X}^*)^{\gamma'})$, where $\gamma'$ varies over the elements of the conjugacy class of $(12)(34)$, is isomorphic to 

\begin{center}
\begin{tabular}{ c    c     c   c  c  c  c                 c c  c                               c    c     c   c  c  c  c         }
 &  &  & $0$ & &                                               \\
  &  & $0$ & & $0$ &  &                              \\
   &  $0$  &  & $\nu$ & & $0$ &                  \\
     $0$  &   & $\nu'$  &  & $\nu'$ &  & $0$ \:, \\
      &  $0$  &  & $\nu$ & & $0$ &               \\
        &  & $0$ & & $0$ &  &                     \\
        &  &  & $0$ & &                             
\end{tabular}
 \end{center}
where $\nu= 2 \Ind_{\Z_2}^{A_5} 1 + 2  \Ind_{\Z_2 \times \Z_2}^{A_5} 1$ and
$\nu' = 2  \Ind_{\Z_2 \times \Z_2}^{A_5} 1$. 

Consider the element $\gamma = (123)$ in $A_5$.
We compute $N(\gamma) = \Z^5/( e_0 + \cdots + e_4 = 0, e_0 = e_1 = e_2) \cong \Z^2 = \Z e_0 + \Z e_3$.
The fixed locus $X^* \cap T^\gamma \subseteq T^\gamma \cong (\C^*)^2$ is a hypersurface which is non-degenerate with respect to its Newton polytope $Q = \conv\{ e_0, e_3, -3e_0 - e_3 \}$. It follows that the closure of $X^* \cap T^\gamma$ in $\tilde{X}^*$  is a genus $2$ curve $C$. For every $x \in C$, $\gamma$ acts on $T_x \tilde{X}^*$ with eigenvalues $\{ 1, e^{2 \pi i /3 } , e^{4 \pi i /3 } \}$ and age $a(\gamma, C) = 1$.  The  image of $C \subseteq \tilde{X}^*$ in $X^*$ contains one singular point of $X^*$, and the pre-image of the singular point in $\tilde{X}^*$ contains a unique $\gamma$-fixed point.

Consider the face $F^*$ of $P^*$ given by the convex hull of the images of $e_0, e_1, e_2$ in $N$, which corresponds to a $2$-dimensional singular cone in the normal fan of $P^*$. 
In this case, $N_{F^*} \cong \Z^2$ is generated by the images of $v_ 1 = e_0 - e_2$ and $v_2 = e_2 - e_3$ in $N$, and $N_{F^*}(\gamma) = \Z v_1 + \Z v_2/ (v_1 + v_2 = 0, 3v_1 = 0) \cong  \Z_3$. 
The $\gamma$-fixed locus $T_{F^*}^\gamma$ consists of three points, two of which lie in $X^*$ and have ages $1$ and $2$ respectively. One verifies that the $\gamma$-fixed locus of the fiber in $\tilde{X}^*$ of each singular point  consists of the disjoint union of two $\P^1$'s (with age $1$) and a point (with ages $1$ and $2$ respectively).

We conclude that $(\tilde{X}^*)^\gamma$ consists of the disjoint union of a genus $2$ curve and $4$ $\P^{1}$'s and $2$ points. 
The
 representation of $A_5$ on  $\bigoplus_{\gamma'} H^*((\tilde{X}^*)^{\gamma'})$, where $\gamma'$ varies over the elements of the conjugacy class of $(123)$, is isomorphic to 

\begin{center}
\begin{tabular}{ c    c     c   c  c  c  c                 c c  c                               c    c     c   c  c  c  c         }
 &  &  & $0$ & &                                               \\
  &  & $0$ & & $0$ &  &                              \\
   &  $0$  &  & $\theta$ & & $0$ &                  \\
     $0$  &   & $\theta'$  &  & $\theta'$ &  & $0$ \:, \\
      &  $0$  &  & $\theta$ & & $0$ &               \\
        &  & $0$ & & $0$ &  &                     \\
        &  &  & $0$ & &                             
\end{tabular}
 \end{center}
where $\theta= 6 \Ind_{\Z_3}^{A_5} 1$ and
$\theta' = 2  \Ind_{\Z_3}^{A_5} 1$. 

Finally, one verifies that $\gamma = (12345)$ acts freely on $\tilde{X}^*$. 
We conclude that the $A_5$-representation $H(\tilde{X}^*, A_5)$ is described by the diamond of representations 
 \begin{center}
\begin{tabular}{ c    c     c   c  c  c  c                 c c  c                               c    c     c   c  c  c  c         }
 &  &  & $0$ & &                                               \\
  &  & $0$ & & $0$ &  &                              \\
   &  $0$  &  & $\Phi$ & & $0$ &                  \\
     $0$  &   & $\Phi'$  &  & $\Phi'$ &  & $0$ \:, \\
      &  $0$  &  & $\Phi$ & & $0$ &               \\
        &  & $0$ & & $0$ &  &                     \\
        &  &  & $0$ & &                             
\end{tabular}
 \end{center}
 where $\Phi = 1  + 4\Ind_{\Z_2}^{A_5} 1  + 8 \Ind_{\Z_3}^{A_5} 1 + 2\Ind_{\Z_2 \times \Z_2}^{A_5} 1$ and
$\Phi' = 1 +  2  \Ind_{\Z_3}^{A_5} 1 +  2  \Ind_{\Z_2 \times \Z_2}^{A_5} 1$.  This completes the verification of Conjecture~\ref{c:Fantechi}. In particular, by taking invariant subspaces of the above representations, we calculate the orbifold Hodge diamond of $\tilde{X}^*/A_5$
to be
 \begin{center}
\begin{tabular}{ c    c     c   c  c  c  c                 c c  c                               c    c     c   c  c  c  c         }
 &  &  & $1$ & &                                               \\
  &  & $0$ & & $0$ &  &                              \\
   &  $0$  &  & $15$ & & $0$ &                  \\
     $1$  &   & $5$  &  & $5$ &  & $1$. \\
      &  $0$  &  & $15$ & & $0$ &               \\
        &  & $0$ & & $0$ &  &                     \\
        &  &  & $1$ & &                             
\end{tabular}
 \end{center}
 
 In this case, a result of Bridgeland, King and Reid \cite{BKRMcKay} implies there exist canonical crepant resolutions of $X/A_5$ and $\tilde{X}^*/A_5$. We briefly recall their result. Let $G$ be a finite group acting effectively on a $3$-dimensional complex manifold $Y$ and assume that for every $g \in G$ and $g$-fixed point $y \in Y$, $g$ acts on $T_y Y$ with determinant $1$. Then $Y/G$ has Gorenstein singularities.  In \cite{NakHilbert}, Nakamura introduced the $G$-Hilbert scheme $G$-$\Hilb(Y)$ as a proposed candidate for a crepant resolution of $Y/G$. Complex valued points of $G$-$\Hilb(Y)$ parametrize $G$-invariant, $0$-dimensional subschemes $Z$ of $Y$, such that the induced representation of $G$ on $H^0(Z, \mathcal{O}_Z)$ is isomorphic to the regular representation $\C[G]$ of $G$. There is a Hilbert-Chow morphism 
 \[
 \tau: G\textrm{-}\Hilb(Y) \rightarrow Y/G,
 \]
 sending a $G$-invariant scheme $Z$ to the class of its support. 
 
 \begin{theorem}\cite[Theorem~1.2]{BKRMcKay}\label{t:GHilb}
With the notation above, $G$-$\Hilb(Y)$ is a crepant resolution of $Y/G$. 
\end{theorem}
 
By the theorem above and Remark~\ref{r:sl}, we conclude that  $A_5$-$\Hilb(X)$ and $A_5$-$\Hilb(\tilde{X}^*)$ are crepant resolutions of $X/A_5$ and $\tilde{X}^*/A_5$ respectively. Hence, the above calculations show that $A_5$-$\Hilb(X)$ and $A_5$-$\Hilb(\tilde{X}^*)$ are Calabi-Yau manifolds with respective mirror Hodge diamonds

\begin{center}
\begin{tabular}{ c    c     c   c  c  c  c                 c c  c                               c    c     c   c  c  c  c         }
 &  &  & $1$ & &                                             & &&      &        &  &  & $1$ & &   \\
  &  & $0$ & & $0$ &  &                                & &&           &  & $0$ & & $0$ &  &\\
   &  $0$  &  & $5$ & & $0$ &                    &  &&         &  $0$  &  & $15$ & & $0$ &\\
     $1$  &   & $15$  &  & $15$ &  & $1$      & &&        $1$  &   & $5$  &  & $5$ &  & $1$. \\
      &  $0$  &  & $5$ & & $0$ &              &&   &       &  $0$  &  & $15$ & & $0$ &  \\
        &  & $0$ & & $0$ &  &                         &&   &          &  & $0$ & & $0$ &  &\\
        &  &  & $1$ & &                               &          &   &&            &  &  & $1$ & & \\
\end{tabular}
 \end{center}

 We observe that the calculations above can be used to verify Conjecture~\ref{c:Fantechi} for any subgroup $\Gamma \subseteq A_5$. In particular, $\Gamma$-$\Hilb(X)$ and $\Gamma$-$\Hilb(\tilde{X}^*)$ are Calabi-Yau manifolds with mirror Hodge diamonds. Below we list the respective Hodge diamonds for all non-trivial proper subgroups of $A_5$ (see, for example, \cite{WikiA5}). Note that when $\Gamma = \langle (12345) \rangle$, the action is free and $X/\Gamma$ and $\tilde{X}^*/\Gamma$ are smooth. 

 \begin{center}
 $\Gamma = \langle (12)(34) \rangle \cong \Z_2$ \\
 \vspace{0.2cm}
\begin{tabular}{ c    c     c   c  c  c  c                 c c  c                               c    c     c   c  c  c  c         }
 &  &  & $1$ & &                                             & &&      &        &  &  & $1$ & &   \\
  &  & $0$ & & $0$ &  &                                & &&           &  & $0$ & & $0$ &  &\\
   &  $0$  &  & $3$ & & $0$ &                    &  &&         &  $0$  &  & $59$ & & $0$ &\\
     $1$  &   & $59$  &  & $59$ &  & $1$      & &&        $1$  &   & $3$  &  & $3$ &  & $1$. \\
      &  $0$  &  & $3$ & & $0$ &              &&   &       &  $0$  &  & $59$ & & $0$ &  \\
        &  & $0$ & & $0$ &  &                         &&   &          &  & $0$ & & $0$ &  &\\
        &  &  & $1$ & &                               &          &   &&            &  &  & $1$ & & \\
\end{tabular}
 \end{center}
 
 \begin{center}
 $\Gamma = \langle (12)(34), (13)(24) \rangle \cong \Z_2 \times \Z_2$ \\
 \vspace{0.2cm}
\begin{tabular}{ c    c     c   c  c  c  c                 c c  c                               c    c     c   c  c  c  c         }
 &  &  & $1$ & &                                             & &&      &        &  &  & $1$ & &   \\
  &  & $0$ & & $0$ &  &                                & &&           &  & $0$ & & $0$ &  &\\
   &  $0$  &  & $7$ & & $0$ &                    &  &&         &  $0$  &  & $41$ & & $0$ &\\
     $1$  &   & $41$  &  & $41$ &  & $1$      & &&        $1$  &   & $7$  &  & $7$ &  & $1$. \\
      &  $0$  &  & $7$ & & $0$ &              &&   &       &  $0$  &  & $41$ & & $0$ &  \\
        &  & $0$ & & $0$ &  &                         &&   &          &  & $0$ & & $0$ &  &\\
        &  &  & $1$ & &                               &          &   &&            &  &  & $1$ & & \\
\end{tabular}
 \vspace{0.2cm}
 \end{center} 
 
  \begin{center}
 $\Gamma = \langle (123) \rangle \cong \Z_3$ \\
 \vspace{0.2cm}
\begin{tabular}{ c    c     c   c  c  c  c                 c c  c                               c    c     c   c  c  c  c         }
 &  &  & $1$ & &                                             & &&      &        &  &  & $1$ & &   \\
  &  & $0$ & & $0$ &  &                                & &&           &  & $0$ & & $0$ &  &\\
   &  $0$  &  & $5$ & & $0$ &                    &  &&         &  $0$  &  & $49$ & & $0$ &\\
     $1$  &   & $49$  &  & $49$ &  & $1$      & &&        $1$  &   & $5$  &  & $5$ &  & $1$. \\
      &  $0$  &  & $5$ & & $0$ &              &&   &       &  $0$  &  & $49$ & & $0$ &  \\
        &  & $0$ & & $0$ &  &                         &&   &          &  & $0$ & & $0$ &  &\\
        &  &  & $1$ & &                               &          &   &&            &  &  & $1$ & & \\
\end{tabular}
 \vspace{0.2cm}
 \end{center} 
 
  \begin{center}
 $\Gamma = \langle (12345) \rangle \cong \Z_5$ \\
 \vspace{0.2cm}
\begin{tabular}{ c    c     c   c  c  c  c                 c c  c                               c    c     c   c  c  c  c         }
 &  &  & $1$ & &                                             & &&      &        &  &  & $1$ & &   \\
  &  & $0$ & & $0$ &  &                                & &&           &  & $0$ & & $0$ &  &\\
   &  $0$  &  & $1$ & & $0$ &                    &  &&         &  $0$  &  & $21$ & & $0$ &\\
     $1$  &   & $21$  &  & $21$ &  & $1$      & &&        $1$  &   & $1$  &  & $1$ &  & $1$. \\
      &  $0$  &  & $1$ & & $0$ &              &&   &       &  $0$  &  & $21$ & & $0$ &  \\
        &  & $0$ & & $0$ &  &                         &&   &          &  & $0$ & & $0$ &  &\\
        &  &  & $1$ & &                               &          &   &&            &  &  & $1$ & & \\
\end{tabular}
 \vspace{0.2cm}
 \end{center}

   \begin{center}
 $\Gamma = \langle (12)(34), (123) \rangle = A_4 \subseteq A_5$ \\
 \vspace{0.2cm}
\begin{tabular}{ c    c     c   c  c  c  c                 c c  c                               c    c     c   c  c  c  c         }
 &  &  & $1$ & &                                             & &&      &        &  &  & $1$ & &   \\
  &  & $0$ & & $0$ &  &                                & &&           &  & $0$ & & $0$ &  &\\
   &  $0$  &  & $7$ & & $0$ &                    &  &&         &  $0$  &  & $29$ & & $0$ &\\
     $1$  &   & $29$  &  & $29$ &  & $1$      & &&        $1$  &   & $7$  &  & $7$ &  & $1$. \\
      &  $0$  &  & $7$ & & $0$ &              &&   &       &  $0$  &  & $29$ & & $0$ &  \\
        &  & $0$ & & $0$ &  &                         &&   &          &  & $0$ & & $0$ &  &\\
        &  &  & $1$ & &                               &          &   &&            &  &  & $1$ & & \\
\end{tabular}
 \vspace{0.2cm}
 \end{center} 
 
    \begin{center}
 $\Gamma = \langle (12)(45), (23)(45) \rangle \cong \Sym_3$ \\
 \vspace{0.2cm}
\begin{tabular}{ c    c     c   c  c  c  c                 c c  c                               c    c     c   c  c  c  c         }
 &  &  & $1$ & &                                             & &&      &        &  &  & $1$ & &   \\
  &  & $0$ & & $0$ &  &                                & &&           &  & $0$ & & $0$ &  &\\
   &  $0$  &  & $5$ & & $0$ &                    &  &&         &  $0$  &  & $33$ & & $0$ &\\
     $1$  &   & $33$  &  & $33$ &  & $1$      & &&        $1$  &   & $5$  &  & $5$ &  & $1$. \\
      &  $0$  &  & $5$ & & $0$ &              &&   &       &  $0$  &  & $33$ & & $0$ &  \\
        &  & $0$ & & $0$ &  &                         &&   &          &  & $0$ & & $0$ &  &\\
        &  &  & $1$ & &                               &          &   &&            &  &  & $1$ & & \\
\end{tabular}
 \vspace{0.2cm}
 \end{center} 
 
     \begin{center}
 $\Gamma = \langle (12)(35), (12345) \rangle = D_{5}$ (dihedral group of order $10$) \\
 \vspace{0.2cm}
\begin{tabular}{ c    c     c   c  c  c  c                 c c  c                               c    c     c   c  c  c  c         }
 &  &  & $1$ & &                                             & &&      &        &  &  & $1$ & &   \\
  &  & $0$ & & $0$ &  &                                & &&           &  & $0$ & & $0$ &  &\\
   &  $0$  &  & $3$ & & $0$ &                    &  &&         &  $0$  &  & $19$ & & $0$ &\\
     $1$  &   & $19$  &  & $19$ &  & $1$      & &&        $1$  &   & $3$  &  & $3$ &  & $1$. \\
      &  $0$  &  & $3$ & & $0$ &              &&   &       &  $0$  &  & $19$ & & $0$ &  \\
        &  & $0$ & & $0$ &  &                         &&   &          &  & $0$ & & $0$ &  &\\
        &  &  & $1$ & &                               &          &   &&            &  &  & $1$ & & \\
\end{tabular}
 \vspace{0.2cm}
 \end{center} 
 
 \begin{remark}
One can similarly verify the corresponding statements when $d= 3$. That is,  let $X$ be the cubic surface $\{ x_0^4 + x_1^4 + x_2^4 + x_3^4 = 0 \} \subseteq \P^3$, and let $\tilde{X}^*$ be a $\Sym_4$-equivariant, crepant resolution of its mirror. Then for any subgroup $\Gamma \subseteq A_4$, the orbifolds  $X/\Gamma$ and $\tilde{X}^*/\Gamma$ satisfy Conjecture~\ref{c:Fantechi}, and have mirror orbifold Hodge diamonds (each with $h^{1,0}_{\orb} = 0$,  $h^{2,0}_{\orb} = 1$ and $h^{1,1}_{\orb} = 20$). Moreover, the $\Gamma$-Hilbert schemes $\Gamma$-$\Hilb(X)$ and 
$\Gamma$-$\Hilb(\tilde{X}^*)$ are crepant resolutions of $X/\Gamma$ and $\tilde{X}^*/\Gamma$ respectively \cite[Theorem~1.2]{BKRMcKay}. 
\end{remark}

\bibliographystyle{amsplain}
\bibliography{alan}

\def\cprime{$'$} \def\cprime{$'$} \def\cprime{$'$} \def\cprime{$'$}
\providecommand{\bysame}{\leavevmode\hbox to3em{\hrulefill}\thinspace}
\providecommand{\MR}{\relax\ifhmode\unskip\space\fi MR }
\providecommand{\MRhref}[2]{%
  \href{http://www.ams.org/mathscinet-getitem?mr=#1}{#2}
}
\providecommand{\href}[2]{#2}
\begin{thebibliography}{10}

\bibitem{WikiA5}
\emph{Subgroup structure of alternating group:a5},
  http://groupprops.subwiki.org/wiki/.

\bibitem{AWEquivariant}
Dan Abramovich and Jianhua Wang, \emph{Equivariant resolution of singularities
  in characteristic {$0$}}, Math. Res. Lett. \textbf{4} (1997), no.~2-3,
  427--433.

\bibitem{ACaToric}
Annette A'Campo-Neuen, \emph{On toric {$h$}-vectors of centrally symmetric
  polytopes}, Arch. Math. (Basel) \textbf{87} (2006), no.~3, 217--226.

\bibitem{AudTorus}
Mich{\`e}le Audin, \emph{Torus actions on symplectic manifolds}, revised ed.,
  Progress in Mathematics, vol.~93, Birkh\"auser Verlag, Basel, 2004.

\bibitem{BatStringy}
Victor Batyrev, \emph{Stringy {H}odge numbers of varieties with {G}orenstein
  canonical singularities}, Integrable systems and algebraic geometry
  ({K}obe/{K}yoto, 1997), World Sci. Publ., River Edge, NJ, 1998, pp.~1--32.

\bibitem{BatNon}
\bysame, \emph{Non-{A}rchimedean integrals and stringy {E}uler numbers of
  log-terminal pairs}, J. Eur. Math. Soc. (JEMS) \textbf{1} (1999), no.~1,
  5--33.

\bibitem{BBMirror}
Victor Batyrev and Lev Borisov, \emph{Mirror duality and string-theoretic
  {H}odge numbers}, Invent. Math. \textbf{126} (1996), no.~1, 183--203.

\bibitem{BCKvSConifold}
Victor Batyrev, Ionu{\c{t}} Ciocan-Fontanine, Bumsig Kim, and Duco van Straten,
  \emph{Conifold transitions and mirror symmetry for {C}alabi-{Y}au complete
  intersections in {G}rassmannians}, Nuclear Phys. B \textbf{514} (1998),
  no.~3, 640--666.

\bibitem{BCKvSMirror}
\bysame, \emph{Mirror symmetry and toric degenerations of partial flag
  manifolds}, Acta Math. \textbf{184} (2000), no.~1, 1--39.

\bibitem{BDStrong}
Victor Batyrev and Dimitrios Dais, \emph{Strong {M}c{K}ay correspondence,
  string-theoretic {H}odge numbers and mirror symmetry}, Topology \textbf{35}
  (1996), no.~4, 901--929.

\bibitem{BNCombinatorial}
Victor Batyrev and Benjamin Nill, \emph{Combinatorial aspects of mirror
  symmetry}, Integer points in polyhedra---geometry, number theory,
  representation theory, algebra, optimization, statistics, Contemp. Math.,
  vol. 452, Amer. Math. Soc., Providence, RI, 2008, pp.~35--66.

\bibitem{BRComputing}
Matthias Beck and Sinai Robins, \emph{Computing the continuous discretely},
  Undergraduate Texts in Mathematics, Springer, New York, 2007, Integer-point
  enumeration in polyhedra.

\bibitem{BoeMirror}
Janko B\"ohm, \emph{Mirror symmetry and tropical geometry}, arXiv:0708.4402.

\bibitem{BorK3}
Ciprian Borcea, \emph{{$K3$} surfaces with involution and mirror pairs of
  {C}alabi-{Y}au manifolds}, Mirror symmetry, {II}, AMS/IP Stud. Adv. Math.,
  vol.~1, Amer. Math. Soc., Providence, RI, 1997, pp.~717--743.

\bibitem{BMString}
Lev Borisov and Anvar Mavlyutov, \emph{String cohomology of {C}alabi-{Y}au
  hypersurfaces via mirror symmetry}, Adv. Math. \textbf{180} (2003), no.~1,
  355--390.

\bibitem{BraRemarks}
Tom Braden, \emph{Remarks on the combinatorial intersection cohomology of
  fans}, Pure Appl. Math. Q. \textbf{2} (2006), no.~4, part 2, 1149--1186.

\bibitem{BraEhrhart}
Benjamin Braun, \emph{An {E}hrhart series formula for reflexive polytopes},
  Electron. J. Combin. \textbf{13} (2006), no.~1, Note 15, 5 pp. (electronic).

\bibitem{BKRMcKay}
Tom Bridgeland, Alastair King, and Miles Reid, \emph{The {M}c{K}ay
  correspondence as an equivalence of derived categories}, J. Amer. Math. Soc.
  \textbf{14} (2001), no.~3, 535--554 (electronic).

\bibitem{COGPPair}
Philip Candelas, Xenia de~la Ossa, Paul Green, and Linda Parkes, \emph{A pair
  of {C}alabi-{Y}au manifolds as an exactly soluble superconformal theory},
  Nuclear Phys. B \textbf{359} (1991), no.~1, 21--74.

\bibitem{CLSCalabi}
Philip Candelas, Monika Lynker, and Rolf Schimmrigk, \emph{Calabi-{Y}au
  manifolds in weighted {${\bf P}_4$}}, Nuclear Phys. B \textbf{341} (1990),
  no.~2, 383--402.

\bibitem{CMSSEquivariant}
Sylvain Cappell, Laurentiu Maxim, Joerg Schuermann, and Julius Shaneson,
  \emph{Equivariant characteristic classes of complex algebraic varieties},
  arXiv:1004.1844.

\bibitem{CRNew}
Weimin Chen and Yongbin Ruan, \emph{A new cohomology theory of orbifold}, Comm.
  Math. Phys. \textbf{248} (2004), no.~1, 1--31.

\bibitem{CheRepresentations}
Gabriel Ch\^enevert, \emph{Representations on the cohomology of smooth
  projective hypersurfaces with symmetries}, arXiv:0908.1748.

\bibitem{DKAlgorithm}
Vladimir Danilov and Askold Khovanski{\u\i}, \emph{Newton polyhedra and an
  algorithm for calculating {H}odge-{D}eligne numbers}, Izv. Akad. Nauk SSSR
  Ser. Mat. \textbf{50} (1986), no.~5, 925--945.

\bibitem{DHVWStrings}
Lance Dixon, Jeffrey Harvey, Cumrun Vafa, and Edward Witten, \emph{Strings on
  orbifolds}, Nuclear Phys. B \textbf{261} (1985), no.~4, 678--686.

\bibitem{DonLefschetz}
Peter Donovan, \emph{The {L}efschetz-{R}iemann-{R}och formula}, Bull. Soc.
  Math. France \textbf{97} (1969), 257--273.

\bibitem{FGOrbifold}
Barbara Fantechi and Lothar G{\"o}ttsche, \emph{Orbifold cohomology for global
  quotients}, Duke Math. J. \textbf{117} (2003), no.~2, 197--227.

\bibitem{FulIntroduction}
William Fulton, \emph{Introduction to toric varieties}, Annals of Mathematics
  Studies, vol. 131, Princeton University Press, Princeton, NJ, 1993, , The
  William H. Roever Lectures in Geometry.

\bibitem{KhoNewton}
Askold Hovanski{\u\i}, \emph{Newton polyhedra, and toroidal varieties},
  Funkcional. Anal. i Prilo\v zen. \textbf{11} (1977), no.~4, 56--64, 96.

\bibitem{KanPfaffian}
Atsushi Kanazawa, \emph{On {P}faffian {C}alabi-{Y}au varieties and mirror
  symmetry}, arXiv:1006.0223.

\bibitem{LehRational}
Gustav Lehrer, \emph{Rational points and {C}oxeter group actions on the
  cohomology of toric varieties}, Ann. Inst. Fourier (Grenoble) \textbf{58}
  (2008), no.~2, 671--688.

\bibitem{MorMirror}
David Morrison, \emph{Mirror symmetry and rational curves on quintic
  threefolds: a guide for mathematicians}, J. Amer. Math. Soc. \textbf{6}
  (1993), no.~1, 223--247.

\bibitem{NakHilbert}
Iku Nakamura, \emph{Hilbert schemes of abelian group orbits}, J. Algebraic
  Geom. \textbf{10} (2001), no.~4, 757--779.

\bibitem{RodPfaffian}
Einar~Andreas R{\o}dland, \emph{The {P}faffian {C}alabi-{Y}au, its mirror, and
  their link to the {G}rassmannian {$G(2,7)$}}, Compositio Math. \textbf{122}
  (2000), no.~2, 135--149.

\bibitem{StaGeneralized}
Richard Stanley, \emph{Generalized {$H$}-vectors, intersection cohomology of
  toric varieties, and related results}, Commutative algebra and combinatorics
  (Kyoto, 1985), Adv. Stud. Pure Math., vol.~11, North-Holland, Amsterdam,
  1987, pp.~187--213.

\bibitem{StaSubdivisions}
\bysame, \emph{Subdivisions and local {$h$}-vectors}, J. Amer. Math. Soc.
  \textbf{5} (1992), no.~4, 805--851.

\bibitem{StaEnumerative}
\bysame, \emph{Enumerative combinatorics. {V}ol. 1}, Cambridge Studies in
  Advanced Mathematics, vol.~49, Cambridge University Press, Cambridge, 1997,
  With a foreword by Gian-Carlo Rota, Corrected reprint of the 1986 original.

\bibitem{YoEquivariant}
Alan Stapledon, \emph{Equivariant {E}hrhart theory}, arXiv:1003.5875, to appear
  in Adv. Math.

\bibitem{YoRepresentations}
\bysame, \emph{Representations on the cohomology of hypersurfaces and mirror
  symmetry}, arXiv:1004.3446.

\bibitem{TevCompactifications}
Jenia Tevelev, \emph{Compactifications of subvarieties of tori}, Amer. J. Math.
  \textbf{129} (2007), no.~4, 1087--1104.

\bibitem{VeyArc}
Willem Veys, \emph{Arc spaces, motivic integration and stringy invariants},
  Singularity theory and its applications, Adv. Stud. Pure Math., vol.~43,
  Math. Soc. Japan, Tokyo, 2006, pp.~529--572.

\bibitem{VoiMiroirs}
Claire Voisin, \emph{Miroirs et involutions sur les surfaces {$K3$}},
  Ast\'erisque (1993), no.~218, 273--323, Journ{\'e}es de G{\'e}om{\'e}trie
  Alg{\'e}brique d'Orsay (Orsay, 1992).

\bibitem{YasMotivic}
Takehiko Yasuda, \emph{Motivic integration over {D}eligne-{M}umford stacks},
  Adv. Math. \textbf{207} (2006), no.~2, 707--761.

\bibitem{ZieLectures}
G{\"u}nter Ziegler, \emph{Lectures on polytopes}, Graduate Texts in
  Mathematics, vol. 152, Springer-Verlag, New York, 1995.

\end{thebibliography}

\end{document}